\documentclass[12pt]{amsart}

\usepackage{enumerate}

\usepackage[active]{srcltx}
\usepackage{nicefrac}

\usepackage{amsthm,amssymb,setspace}
\usepackage[pagebackref,colorlinks,linkcolor=red,citecolor=blue,urlcolor=blue,hypertexnames=true]{hyperref}
\usepackage{amsrefs, amsfonts}
\usepackage[matrix, arrow]{xy}

\usepackage{makecell}

\newcommand{\C}{\mathbb C}

\newcommand{\suppe}{{\rm supp}}
\newcommand{\suppg}{{\it supp}}
\newcommand{\fpc}{{\rm fpc}}
\newcommand{\fix}{{\rm Fix}}
\newcommand{\eg}{\text{e.g.\,}}
\newcommand{\ie}{\text{i.e.,\;}}
\newcommand{\IChar}{{\rm IChar}}
\newcommand{\ICh}{{\rm ICh}}
\newcommand{\IRS}{{\rm IRS}}
\newcommand{\Sub}{{\rm Sub}}

\theoremstyle{plain}
\newtheorem{theorem}{Theorem}[section]

\newtheorem{thma}{Theorem}

\newtheorem{lemma}[theorem]{Lemma}
\newtheorem{proposition}[theorem]{Proposition}

\theoremstyle{definition}
\newtheorem{question}[theorem]{Question}
\newtheorem{definition}[theorem]{Definition}

\theoremstyle{remark}

\newtheorem{remark}[theorem]{Remark}

\newcommand{\comment}[1]{\textcolor{blue}{\bf [#1]}}

\begin{document}

\onehalfspace

\title{A character approach to the ISR property}

\author{Artem Dudko}
\address{Institute of Mathematics of the Polish Academy of Sciences, Ul. S'niadeckich 8, 00-656 Warszawa, Poland}
\address{ B.Verkin Institute for Low Temperature Physics and Engineering of the National Academy of Sciences of Ukraine, 47 Nauky Ave., Kharkiv, 61103, Ukraine}
\email{adudko@impan.pl}

\author{Yongle Jiang}
\address{School of Mathematical Sciences, Dalian University of Technology, Dalian, 116024, China}
\email{yonglejiang@dlut.edu.cn}

\subjclass[2020]{Primary 20C15,  46L10; Secondary 43A35, 22D25}
\keywords{Characters, Invariant von Neumann subalgebras rigidity, Approximately finite groups, Character rigid groups}
\date{\today}

\begin{abstract} We develop a character approach to study the invariant von Neumann subalgebras rigidity property (abbreviated as the ISR property) introduced in Amrutam-Jiang's work. First,
we introduce the non-factorizable regular character property for groups and show that this implies the ISR property for any infinite ICC groups with trivial amenable radical.
Various examples are shown to have this property. Second, we apply known classification results on indecomposable characters to show approximately finite groups have the ISR property. Based on this approach, we also construct non-amenable groups with the ISR property while having non-trivial amenable radical or without the non-factorizable regular character property. 
\end{abstract}

\maketitle

%\tableofcontents
\section{Introduction}
Recently, there has been a growing interest in studying the structure of $G$-invariant von Neumann subalgebras in the group von Neumann algebra $L(G)$ for a countable discrete group $G$ \cites{ab,bru,kp,aj,cds,aho,jz}. A common question studied in these works is searching for groups with a new rigidity property, \ie the so-called invariant von Neumann subalgebras rigidity property (ISR property for short), which was introduced in \cite{aj} after the first class of groups with this property appeared in an earlier work of Kalantar-Panagopoulos \cite{kp}. Recall that a countable disrecte group $G$ has the ISR property if every $G$-invariant (under the conjugation) von Neumann subalgebra in $L(G)$ is equal to $L(H)$ for some normal subgroup $H\subseteq G$. 

There are at least two reasons for us to pursue this line of research. First, maximal abelian $G$-invariant von Neumann  subalgebras are prototypes of Cartan subalgebras in $L(G)$. Hence, the study of classifying $G$-invariant von Neumann subalgebras in $L(G)$ may shed light on the famous open question of classifying Cartan subalgebras in $L(G)$ for certain classes of groups, see e.g., \cite[Problem V]{ioana_icm}, \cite[Corollary 4.4]{aj} and \cite[Corollary 2.8]{kp}. Second, $G$-invariant von Neumann subalgebras in $L(G)$ are nothing but the simplest, that is, the Dirac type invariant random von Neumann subalgebras, a new concept introduced in \cite{aho}. Note that this concept is a von Neumann algebra analogue of the so-called invariant random subgroups (IRS for short) in group theory, which is a quite hot topic and has been extensively explored; see \cites{agv,gel} and references therein. Naturally, it is interesting to ask to what extent the theory of IRS can be transferred to the study of  invariant random von Neumann subalgebras. For some initial results along this direction, see \cite{aho}.

In this paper, we develop a character approach to study invariant von Neumann subalgebras and establish the ISR property \cite{aj} for several remarkable classes of both amenable and non-amenable groups not covered in previous works \cites{kp,aj,cds,jz}. Here, by a character we mean a normalized conjugation invariant positive definite function on a countably infinite discrete group $G$. We refer the reader to the book \cite{bd_book} for an overview of the long history on the study of this concept. Quite recently,  characters play a key role in studying (local) Hilbert-Schmidt stability problem, see e.g., \cites{es,fgs,hs_1,hs_2,lv}.

There are two main aspects in our character approach. First, we introduce the new concept of non-factorizable regular character property for a countable discrete group $G$ (Definition \ref{def: NFRC}). With the aid of this new concept and the recent striking theorem due to Amrutam-Hartman-Oppelmayer \cite{aho}, we are able to establish a sufficient condition for the ISR property. 

\begin{thma}[Theorem \ref{theorem: NFRC consequences}]
Let $G$ be an ICC group. If $G$ has non-factorizable regular character property, then every $G$-invariant subfactor of $L(G)$ is equal to $L(H)$ for some normal subgroup $H$ of $G$. If $G$ is further assumed to be non-amenable with trivial amenable radical, then $G$ has the ISR property.
\end{thma}

Various examples of groups are shown to have the non-factorizable regular character property, including many ICC simple groups e.g., character rigid groups such as $PSL_n(\mathbb{Q})$ for $n\geq 2$ and $L(Alt)$-groups (Proposition \ref{prop: examples of NSFC-simple groups}) and certain relative ICC extensions (Proposition \ref{prop: examples of NFRC-relative ICC extension of simple groups}); Besides, we also show many charfinite groups introduced in \cite{bbhp}, e.g., $SL_3(\mathbb{Z})$, also have this non-factorizable regular character property, see  Proposition \ref{prop: SLnZ}, Proposition \ref{prop: examples of NFRC-charfinite groups_1} and Proposition \ref{prop: examples of NFRC-charfinite groups_2}, extending the work \cite{kp} with a completely different and more conceptual approach. 
In fact, we believe that this non-factorizable regular character property is actually prevalent among ICC groups if some obvious obstruction, e.g., item (2) in Proposition \ref{prop: NFRC basic properties} vanishes. It is also of independent interest, see Proposition \ref{prop: hanfeng's proof} for a connection to ergodic theory.

Actually, even for ambient groups without this non-factorizable property, we can still combine the analysis of normal subgroup structure with this property shared by certain well located subgroups to deduce the ISR property for the ambient groups, see e.g., Proposition \ref{prop: product of n groups with character rigidity has the ISR property}, Proposition \ref{prop: ISR for certain generalized wreath product groups}. 

Recall that all previous known examples of non-amenable groups with the ISR property have trivial amenable radical.
By exploring the above strategy, we show that the class of non-amenable groups with the ISR property is more diverse than what we expected.  

\begin{thma}[Theorem \ref{prop: example of nonamenable groups with ISR but nontrivial amenable radical}]
There exists a non-amenable group with both non-trivial amenable radical and the ISR property.
\end{thma}

The other aspect in this character approach is that we combine techniques in \cites{aj,jz} and classification work on indecomposable characters in \cites{dm_ggd2014, dm_manuscript,thoma_mz} to establish the ISR property for a wide class of amenable groups with dynamical origin, greatly extending the class of amenable groups with the ISR property.  Indeed, amenable groups with the ISR property was not known until the quite recent work \cite{jz}.

\begin{thma}[Theorem \ref{thm: approximately finite groups have the ISR property}]
Full groups of simple Bratteli diagrams (known also as approximately finite groups) have the ISR property.
\end{thma}
\noindent
We expect the proof strategy used in Section \ref{sec:approximately finite} could also be adapted to deal with other class of groups with dynamical origin, see the discussion in Section \ref{sec: more question}.

To summarize, we give an overview of the classes of groups that are shown to have the ISR property or the weaker property that every invariant von Neumann subfactor of $L(G)$ is a subgroup subfactor based on our character approach.
\begin{itemize}
\item{} Groups with the non-factorizable regular character property, e.g.,
\begin{itemize}
\item Several classes of infinite ICC groups arising from simple groups (Proposition \ref{prop: examples of NSFC-simple groups}); this includes $a)$ character rigid groups and $b)$ $L(Alt)$-groups;
\item{} Many charfinite groups (Section \ref{sec: OC property}, Proposition \ref{prop: NFRC basic properties}); this includes $SL_n(\mathbb{Z})$ for odd $n\geqslant 3$ (Proposition 
\ref{prop: SLnZ}).
\end{itemize}
\item{} Approximately finite groups (Section \ref{sec:approximately finite});
%\item{} Finitary automorphisms of a regular rooted tree (Proposition \ref{prop: weakly branch groups have the ISR property}).
\item{} Certain infinite groups without the non-factorizable regular character property, including 
\begin{itemize}
\item Two classes of groups with nice normal subgroup structures, \ie direct product of non-amenable chracter rigid groups (Proposition \ref{prop: product of n groups with character rigidity has the ISR property}); and  certain generalized wreath product groups (Proposition \ref{prop: ISR for certain generalized wreath product groups}).
\item An infinite group with non-trivial amenable radical (Theorem \ref{prop: example of nonamenable groups with ISR but nontrivial amenable radical}).
\end{itemize}
\end{itemize}

This paper is organized as follows. We recall some basic concepts on characters, group von Neumann algebras and the ISR property in Section \ref{sec: preliminaries}. Section \ref{sec: OC property} is devoted to a detailed study of the non-factorizable regular character property, including its definition, relation to the ISR property, \ie Theorem \ref{theorem: NFRC consequences} and various examples with this property.   
%In Section \ref{sec: weakly branch} we study invariant von Neumann subalgebras of weakly branch groups.
 Then we prove approximately finite groups have the ISR property in Section \ref{sec:approximately finite}. 
We also construct  non-amenable groups with the ISR property but without the non-factorizable regular character property in Subsection \ref{subsection: ISR groups without the factorizable property}   or having non-trivial amenable radical in Subsection \ref{subsection: non-trivial amenable radical}. Finally, we collect some questions arising from this study in Section \ref{sec: more question}.

\paragraph{\textbf{Acknowledgements}} 

A. D. acknowledges the funding by Long-term program of support of the Ukrainian
research teams at the Polish Academy of Sciences carried out in collaboration with the
U.S. National Academy of Sciences with the financial support of external partners. 

Y. J. is partially supported by National Natural Science Foundation of China (Grant No. 12471118). The authors thank Dr. Tattwamasi Amrutam and Prof. Adam Skalski for several helpful discussion.

\section{Preliminaries}
\label{sec: preliminaries}

\subsection{Characters}

In this subsection, we record some basic facts about the characters. More detailed treatment can be found in \cite{bd_book}.

\begin{definition}
Let $G$ be a countable group. A {\it character} on $G$ is a function $\phi: G\rightarrow \mathbb{C}$ such that
\begin{itemize}
\item[(1)] $\phi$ is \emph{positive definite}, \ie \[\sum_{i=1}^n\sum_{j=1}^n\overline{\alpha_i}\alpha_j\phi(g_i^{-1}g_j)\geq 0\]
for all $n\in\mathbb{N}$ and any elements $g_1,\ldots, g_n\in G$ and any $\alpha_1,\ldots, \alpha_n\in\mathbb{C}$,
\item[(2)] $\phi$ is \emph{conjugation-invariant}, \ie $\phi(hgh^{-1})=\phi(g)$ for all $g, h\in G$, and
\item[(3)] $\phi$ is \emph{normalized}, \ie $\phi(e)=1$.
\end{itemize} 
We say that a character $\phi$ on $G$ is {\it indecomposible} or {\it extremal} if it cannot be written in the form $\phi=\alpha\phi_1+(1-\alpha)\phi_2$, where $0<\alpha<1$ and $\phi_1,\phi_2$ are distinct characters.
\end{definition}

Note that the space of characters on $G$ endowed with the topology of pointwise convergence makes it a compact convex subset of $\ell^{\infty}(G)$. Then indecomposable characters are just the extreme points of the space of characters.

\subsection{Group von Neumann algebras and the ISR property}

In this subsection, we recall some basic facts about group von Neumann algebras \cite{ap_book} and the ISR property as introduced in \cite{aj}.

\begin{definition}
Let $G$ be a countable group. Then {\it the group von Neumann algebra generated by $G$}, denoted by $L(G)$, is defined as the closure of the linear span of all operators $u_g$, $g\in G$, in $B(\ell^2(G))$ under the weak operator topology, where $u_g\in\mathcal{U}(\ell^2(G))$ is the unitary operator defined by $u_g(\delta_s)=\delta_{gs}$ for all $s\in G$. 
\end{definition}

Recall that \emph{the canonical trace} $\tau$ on $L(G)$ is defined by $\tau(x)=\langle x\delta_e,\delta_e\rangle$ for any $x\in L(G)$. Moreover, $\tau$ is faithful, and hence the map $L(G)\ni x\mapsto x\delta_e\in \ell^2(G)$ is injective. 
For any $x\in L(G)$, we may write $x=\sum_{g\in G}c_gu_g$ for its \emph{Fourier expansion}, \ie $x\delta_e=\sum_{g\in G}c_g\delta_g$ for some $(c_g)_{g\in G}\in\ell^2(G)$.
Set $\text{supp}(x)=\{g\in G: c_g\neq 0\}$ for its \emph{support}. We often simply write $g$ for $u_g$ if no confusion arises.

Note that for any subgroup $H$ of $G$, we could identify $L(H)$ as a von Neumann subalgebra of $L(G)$, \ie $L(H)=\{a\in L(G):~\text{supp}(x)\subseteq H\}$. If $H$ is a normal subgroup, then clearly $L(H)$ is $G$-invariant in the sense that $u_gL(H)u_g^*=L(H)$ for all $g\in G$. Surprisingly, it was proved in \cite{kp} that for certain higher rank lattice subgroups, e.g., $G=SL_3(\mathbb{Z})$, there are no other $G$-invariant von Neumann subalgebras in $L(G)$. To record this rigidity feature, together with Amrutam, we introduced the so-called ISR property in \cite{aj}. 
\begin{definition}
Let $G$ be a countable group. Then we say $G$ has the {\it invariant von Neumann subalgebras rigidity property}, abbreviated as the {\it ISR property}, if for any $G$-invariant von Neumann subalgebra $P\subseteq L(G)$, then there is some normal subgroup $H\lhd G$ such that $P=L(H)$.
\end{definition}

Let $P\subseteq L(G)$ be a $G$-invariant von Neumann subalgebra. Denote by $E: (L(G),\tau)\rightarrow P$ the trace $\tau$-preserving conditional expectation onto $P$, for its definition and basic properties, see \cite[\S 9.1]{ap_book}. We record some properties of $E$ needed for the study of the ISR property.
\begin{proposition}\label{prop: basic properties of c.e. related to the ISR property}
The following hold true.
\begin{itemize}
    \item[(1)] $\tau(E(g)s)=\tau(E(g)E(s))=\tau(gE(s))$ for all $s, g\in G$.
    \item[(2)] $sE(g)s^{-1}=E(sgs^{-1})$ for all $s, g\in G$.
\end{itemize}
\end{proposition}
\begin{proof}
(1) 
\begin{align*}
&\tau(E(g)s)\\
&=\tau(E(E(g)s))~(\text{since $E$ is $\tau$-preserving, \ie $\tau\circ E=\tau$})\\
&=\tau(E(g)E(s))~(\text{since $E$ is $P$-bimodular, \ie $E(p_1xp_2)=p_1E(x)p_2$, $\forall p_1, p_2\in P$, $\forall x\in L(G)$})\\
&=\tau(E(gE(s)))~(\text{since $E$ is $P$-bimodular})\\
&=\tau(gE(s)).~(\text{since $E$ is $\tau$-preserving})
\end{align*}
(2) By the definition of $E$ (see the proof of Theorem 9.1.2 in \cite{ap_book}), we know that for any $x\in L(G)$, $E(x)$ is the unique element in $L(G)$ such that $\langle x-E(x), y\rangle=0$ for all $y\in P$, where $\langle a, b\rangle:=\tau(b^*a)$. In other words, $E(x)$ is the image of $x\in L(G)\subseteq \ell^2(G)$ under the orthogonal projection $\ell^2(G)\twoheadrightarrow L^2(P,\tau)$. 

To prove (2), it suffices to check that $\langle sgs^{-1}-sE(g)s^{-1}, y\rangle=0$ for all $y\in P$. Indeed, 
\begin{align*}
\langle sgs^{-1}-sE(g)s^{-1}, y\rangle&=\tau(y^*(sgs^{-1}-sE(g)s^{-1}))\\
&=\tau(s^{-1}y^*s(g-E(g)))\\
&=\langle g-E(g), s^{-1}ys\rangle=0,
\end{align*}
where to get the last equality, we have used the fact that $P$ is $G$-invariant and hence $s^{-1}ys\in P$, thus the definition of $E(g)$ entails the equality.
\end{proof}

Another fact we will use is the following simple lemma to control relative commutant, which can be traced back to Dixmier's work on MASAs (maximal abelian von Neumann subalgebras), see e.g., \cite{ss_book}. In fact, its proof is a natural generalization of the well-known fact that $G$ is ICC iff $L(G)$ is a factor.

\begin{proposition}\label{prop: control supports for elements in relative commutants}
Let $H\subseteq G$ be countable discrete groups and $x\in L(H)'\cap L(G)$ be any element. Then $ \textrm{supp}(x)\subseteq \{g\in G: \sharp\{tgt^{-1}: t\in H\}<\infty\}$.
 \end{proposition}
\begin{proof}
Let $x=\sum_{g\in G}c_gu_g$ be the Fourier expansion of $x$. 

Since $x\in L(H)'$, we deduce that $u_txu_t^*=x$ for all $t\in H$. By the uniqueness of Fourier expansion, this is equivalent to $c_g=c_{tgt^{-1}}$ for all $g\in G$ and $t\in H$.

Assume that $\sharp\{tgt^{-1}: t\in H\}=\infty$ for some $g\in G$. Set $I_g=\{tgt^{-1}: t\in H\}$. 
Then \begin{align*}
    \infty>\sum_{s\in G}|c_s|^2\geq \sum_{s\in I_g}|c_s|^2=\sum_{s\in I_g}|c_g|^2=\sharp I_g\cdot |c_g|^2.
\end{align*}
Thus $c_g=0$ since $\sharp I_g=\infty$.
\end{proof}

To establish some group $G$ have the ISR property, we have the following simple observation, which was already used in \cite{aj}.
\begin{proposition}\label{prop: criterion on P=L(H)}
Let $P\subseteq L(G)$ be a $G$-invariant von Neumann subalgebra. Then $P=L(H)$ for some subgroup $H\lhd G$ iff $E(u_g)\in\mathbb{C}u_g$ for all $g\in G$, where $E: L(G)\rightarrow P$ denotes the unique trace $\tau$-preserving conditional expectation on $P$.
\end{proposition}
\begin{proof}
The ``only if" direction: this is clear since $E(u_g)=u_g$ or $0$ depending on whether $g\in H$ or not.

The ``if" direction: set $H=\{g\in G: E(u_g)\neq 0\}$.
We aim to check that $H\lhd G$ and $P=L(H)$.

First, we show $H$ is a subgroup.
$E(u_e)=E(id)=id\neq 0$, hence $e\in H$. Moreover, $E(u_s)\neq 0\Rightarrow E(u_{s^{-1}})=E(u_s^*)=E(u_s)^*\neq 0$, where to get the last equality, we have used the fact that $E$ is a positive map.

Assume that $s, t\in H$, \ie $E(u_s)\neq 0$ and $E(u_t)\neq 0$. By our assumption, $E(u_s)=c_su_s$ and $E(u_t)=c_tu_t$ for two non-zero numbers $c_s, c_t$. Thus, $E(u_{st})=E(u_su_t)=E(\frac{E(u_s)}{c_s}\frac{E(u_t)}{c_t})=\frac{1}{c_sc_t}E(E(u_s)E(u_t))=\frac{1}{c_sc_t}E(u_s)E(u_t)=u_su_t=u_{st}$. Hence $st\in H$.

Clearly, the subgroup $H$ is normal since $P$ is $G$-invariant.
We are left to show $P=L(H)$.

By the definition of $H$, we have that for any $h\in H$, $u_h\in\mathbb{C}E(u_h)\subseteq P$. Hence $L(H)\subseteq P$ since $L(H)$ is generated by $H$.
Next, given any $x\in P$, write $x=\sum_{g\in G}c_gu_g$ for its Fourier expansion. Note that $L(H)=\{x\in L(G): \text{supp}(x)\subseteq H\}$. Thus to show $x\in L(H)$, it suffices to check that
$\forall g\not\in H$, we have $g\not\in\text{supp}(x)$, \ie $\tau(xu_{g}^{-1})=0$.
Indeed,
$\tau(xu_g^{-1})=\tau(E(xu_{g}^{-1}))=\tau(xE(u_g^{-1}))=\tau(xE(u_g)^*)=0$ since $E(u_g)=0$ by the definition of $H$ and $g\not\in H$.
\end{proof}

While studying invariant subalgebras rigidity using characters we will also use the following statement.
\begin{lemma}\label{lemma: vanishing character from many conjugates} Let $G$ be a countable group and $\chi$ be a character on $G$. Assume that $g\in G$ such that for any $n\in\mathbb N$ there exist elements $g_1,\ldots, g_n\in G$ conjugated to $g$ with $|\chi(g_ig_j^{-1})|\leqslant\epsilon$ for every $i<j$. Then $|\chi(g)|\leqslant \epsilon ^{1/2}$.
\end{lemma} \noindent The proof can be easily derived from the proof of Proposition 2.7 in \cite{dm_ggd2014}.
\begin{proof} Denote by $(\pi,\xi)$ the GNS construction associated to $\chi$. Let $\xi_i=\pi(g_i)\xi$, $i=1,\ldots,n$. Then
$$\|\xi_1+\ldots+\xi_n\|^2=n+\sum\limits_{i\neq j}\chi(g_ig_j^{-1})\leqslant n+n(n-1)\epsilon.$$ It follows that
$$|\chi(g)|=\tfrac{1}{n}|\langle\sum\limits_{i=1}^n\xi_i,\xi\rangle|\leqslant\tfrac{1}{n}(n+n(n-1)\epsilon)^{1/2}.$$ Taking limit when $n\to\infty$ we obtain the desired estimate.
\end{proof}

\section{The non-factorizable regular character property}\label{sec: OC property}
In this section, we introduce the concept of non-factorizable regular character property for countable discrete groups and undertake a detailed study on it, including its basic properties and various examples. 

Recall that in \cite{jz}, the finitary permutation group $S_{\infty}$ on $\mathbb{N}$ was shown to have the ISR property, which is the first known infinite amenable group with this property. The key observation used in the proof is that for any character $\phi$ on $S_{\infty}$, we have $\phi((1~2~3))>0$ unless $\phi\equiv \delta_e$, where $(1~2~3)$ denotes the 3-cycle on $1,2,3$. This observation is recorded as Corollary 2.4 in \cite{jz} and follows from  Thoma's celebrated theorem on classification of indecomposable characters on $S_{\infty}$ \cite{thoma_mz} (see also \cite{KerovVershik-AsymptoticTheorey-81} and \cite{Okunkov-CharactersSymmetric-97}). Motivated by this key property that $S_{\infty}$ has, we introduce the  following concept for all countable groups.

 \begin{definition}\label{def: NFRC}
Let $G$ be a countable group. We say the regular character of $G$ is \emph{non-factorizable} (or that $G$ has {\it non-factorizable regular character}) if for any characters $\phi,\psi$ on $G$ such that  $\phi(s)\psi(s)=0$ for all $e\neq s\in G$ one has either $\phi\equiv \delta_e$ or $\psi\equiv \delta_e$, where $\delta_e$ denotes the Dirac function on the identity element $e$ in $G$.
\end{definition}
Let us record several basic properties of groups having non-factorizable regular characters.
%{\color{green} May be, would be better to move it before Theorem \ref{theorem: NFRC consequences}, since it uses Proposition \ref{prop: NFRC basic properties}?}{\color{blue}  Sure, feel free to change the position as you suggested.}
%%%%%%%%%%%%%copied%%%%%%%%%%%%%%
\begin{proposition}\label{prop: NFRC basic properties}
If $G$ has non-factorizable regular character, then
\begin{itemize}
\item[(1)] every non-trivial normal subgroup of $G$ also has non-factorizable regular character.
\item[(2)] $G$ contains no non-trivial (commuting) normal subgroups $H$ and $K$ with trivial intersection.
\end{itemize}
\end{proposition}
\begin{proof}
(1) Let $\{e\}\neq H\lhd G$ and $\phi,\psi$ be two characters on $H$ with $\phi(s)\psi(s)=0$ for all $e\neq s\in H$. Then we may extend them trivially to get characters on $G$, say $\phi',\psi'$. That is, $\phi'$ agrees with $\phi$ on $H$ and vanishes outside of $H$. Similarly for $\psi'$. Then $\phi'(g)\psi'(g)=0$ for all $e\neq g\in G$. Hence since $G$ has non-factorizable regular character, we deduce that $\phi'\equiv \delta_e$ or $\psi'\equiv \delta_e$. Clearly, this implies $\phi\equiv \delta_e$ or $\psi\equiv \delta_e$.

(2) Assume that $H\lhd G$ and $K\lhd G$ with $H\cap K=\{e\}$, then set $\phi=\delta_H$ and $\psi=\delta_K$. Clearly, $\phi(s)\psi(s)=0$ for all $e\neq s\in G$. Hence $H$ or $K$ is trivial depending on $\phi\equiv\delta_e$ or $\psi\equiv \delta_e$.
\end{proof}

As we will show in this section, the class of groups having non-factorizable regular character is sufficiently rich. %\textcolor{green}{Add a proposition, gathering examples}.
%\textcolor{blue}{Below I added the proof for the theorem and re-organize the following subsections to put examples appearing first.}
One of the main results of this section is the following:
\begin{theorem}\label{theorem: NFRC consequences}
Let $G$ be an ICC group. Assume $G$ has non-factorizable regular character. Then every $G$-invariant subfactor of $L(G)$ is equal to $L(H)$  for some normal subgroup $H$ of $G$. If we further assume that $G$ is non-amenable and has trivial amenable radical, then $G$ has the ISR property.
\end{theorem}

\begin{proof}
(1)  Let $G$ be an ICC group with non-factorizable regular character.  Let $P$ be a $G$-invariant subfactor of $L(G)$. 
By the proof of \cite[Theorem 3.1]{cds}, there exists some normal subgroup $H\lhd G$ such that $P\vee (P'\cap L(H))=L(H)$. More precisely, \[H=\{g\in G: g=a_gb_g~\text{for some $a_g\in \mathcal{U}(P)$ and $b_g\in\mathcal{U}(P'\cap L(G))$}\}.\]

If $H=\{e\}$, then $P=\mathbb{C}$, we are done. 
Without loss of generality, we may assume $H\neq \{e\}$.

From this decomposition, we define $\phi(g)=\tau(E(g)g^{-1})$ and $\psi(g)=\tau(E'(g)g^{-1})$ for all $g\in H$, where $E: L(G)\rightarrow P$ and $E': L(G)\rightarrow P'\cap L(H)$ denote the trace preserving conditional expectations. It is known that $\phi,\psi$ are characters on $H$, say by \cite[Proposition 3.2]{jz}. 

By the definition of $H$, it is not hard to see that $E(s)=a_s\tau(b_s)$ and $E'(s)=b_s\tau(a_s)$ for all $s\in H$. We explain the proof for the 1st equality, and the other one can be proved similarly. %{\color{green} For the first one, can we just write $E(s)=E(a_sb_s)=a_sE(b_s)$, since $a_s\in P$, and argue that $E(b_s)$ belongs to the center of $P$, and therefore is equal to scalar? Though, it seems that using similar idea for 
%$E'(s)$ would require that $P'\cap L(H)$ is a factor.} {\color{blue} that is a subtle point. First, let me make sure that the written argument still works to show $E'(s)=\tau(a_s)b_s$ right? since we need to check $\langle s-\tau(a_s)b_s, y\rangle=0$ for all $y\in P'\cap L(H)$, equivalently, we need $0=\tau((a_s-\tau(a_s)b_sy^*)=\tau(a_s-\tau(a_s)\tau(b_sy^*)$, which clearly holds. Here, the last equality holds since $P$ is a subfactor implies $\tau$ splits as a product on simple tensors as we wrote and the fact that $b_sy^*\in P'\cap L(H)$.

%While to use the argument you suggested, we need to see $E'(a_s)\in P'\cap P$, while it is only clear that it lies in the center of $P'\cap L(H)$ without using the fact that $\tau$ splits as a product on a simple tensor. I think the proof of this splitting property circumvents the issue that $P'\cap L(H)$ must be a factor... For now, I could not see more than you.}
 Indeed, we just need to verify that $\langle s-a_s\tau(b_s), x\rangle=0$ for all $x\in P$, \ie 
\begin{align*}
\tau((s-a_s\tau(b_s))x^*)=\tau((a_sb_s-a_s\tau(b_s))x^*)=\tau(x^*a_s(b_s-\tau(b_s)))=\tau(x^*a_s)\tau(b_s-\tau(b_s))=0,
\end{align*}
where for the 2nd last equality, we have used the fact that $\tau$ on $L(H)=P\bar{\otimes}(P'\cap L(H))$ satisfies that $\tau(xy)=\tau(E(x)y)=\tau(xE(y))=\tau(x)\tau(y)$ for all $x\in P$ and $y\in P'\cap L(H)$ which follows from $E(y)\in P\cap P'=\mathbb{C}$ and hence $E(y)=\tau(y)$.

Therefore, we have $\phi(s)=\tau(a_s\tau(b_s)b_s^*a_s^*)=|\tau(b_s)|^2$ and similarly, $\psi(s)=|\tau(a_s)|^2$. Note that $0=\tau(s)=\tau(a_s)\tau(b_s)$ for all $e\neq s\in H$. 
Thus $\phi(s)\psi(s)=0$ for all $e\neq s\in H$. As a normal subgroup of $G$, $H$ also has non-factorizable regular character by item (2) in Proposition \ref{prop: NFRC basic properties}, therefore, we deduce that $\phi\equiv\delta_e$ or $\psi\equiv\delta_e$. 

If $\phi\equiv\delta_e$, then $0=\phi(s)=\tau(E(s)s^{-1})=\tau(E(s)E(s)^*)=\left\|E(s)\right\|_2^2$, \ie $E(s)=0$ for all $e\neq s\in H$. Thus $P=\mathbb{C}$.

If $\psi\equiv\delta_e$, then by a similar argument as above, we deduce that $E'(s)=0$ for all $e\neq s\in H$. Thus
$P'\cap L(H)=\mathbb{C}$ and hence $L(H)=P\bar{\otimes}\mathbb{C}=P$. 

This finishes the proof of the first part.

(2) In addition, assume now that $G$ is non-amenable and has trivial amenable radical. Let $P$ be a $G$-invariant von Neumann subalgebra. Then its center $\mathcal{Z}(P)$ is also a $G$-invariant abelian and hence amenable von Neumann subalgebra. Since $G$ has trivial amenable radical, by \cite[Theorem A]{aho} we get that $\mathcal{Z}(P)=\mathbb{C}$, \ie $P$ is a subfactor. Thus, the above proof shows that $P=L(H)$ for some normal subgroup $H$ of $G$. 
\end{proof}

In view of Theorem \ref{theorem: NFRC consequences} and the work in \cite{jz}, it is natural to ask whether ICC amenable groups with non-factorizable regular character has the ISR property. A naive approach to show this might be to argue that for amenable ambient groups $G$,  every abelian $G$-invariant von Neumann subalgebra  $P\subseteq L(G)$ is contained in $L(H)$ for some  abelian normal subgroup $H\leq G$. However, the following shows this does not hold in general.

\begin{proposition}\label{prop: abelian invariant subalgebras}
Let $G=A_5\wr \mathbb{Z}$, where $A_5$ is the alternating group on five letters. Then $G$ has no non-trivial abelian normal subgroups and $G$ does not have the ISR property.
\end{proposition}
\begin{proof}
As indicated in the paragraph after Corollary 3.3 in \cite{cds}, we may consider $A=\mathcal{Z}(L(\oplus_{\mathbb{Z}}A_5))$. Clearly, $A$ is $G$-invariant but is not of the form  $L(H)$ for some normal subgroup in $G$. 

Next, we check that $G$ contains no non-trivial abelian normal subgroups.

Let $H$ be a non-trivial abelian normal subgroup in $G$.
Note that $A_5$ is a simple group and hence $H\cap \oplus_{\mathbb{Z}}A_5=\{e\}$ by considering the coordinate projection maps. This implies that $H=\{(\phi(s), s):~s\in\pi(H)\}$, where $\pi: G\twoheadrightarrow G/{\oplus_{\mathbb{Z}}A_5}\cong \mathbb{Z}$ is the projection 
and $\phi: \pi(H)\rightarrow \oplus_{\mathbb{Z}}A_5$ is a group homomorphism. Then from $H\ni g(\phi(s), s)g^{-1}=(\sigma_g(\phi(s)), gsg^{-1})=(\sigma_g(\phi(s)),s)$ 
we deduce that $\phi(s)=\phi(gsg^{-1})=\sigma_g(\phi(s))$ for all $g\in \mathbb{Z}$ and all $s\in \pi(H)$, where $\sigma: \mathbb{Z}\rightarrow Aut(\oplus_{\mathbb{Z}}A_5)$ is the group homomorphism defined by shifting coordinates. Clearly, this implies that $\phi(s)=e$ for all $s\in \mathbb{Z}$, \ie $H\subseteq \mathbb{Z}$. Now, it is clear to 
check that this implies $H$ is trivial since $\sigma_s$ does not fix any non-trivial elements in $\oplus_{\mathbb{Z}}A_5$ for any non-trivial $s\in \mathbb{Z}$.
\end{proof}

Since it is still unclear what other conditions on ICC amenable groups would be needed for deducing the ISR property, we ask the following natural question.

\begin{question}
Is there any ICC infinite amenable group without the ISR property but having non-factorizable regular characters?
\end{question}
\noindent 
In view of Theorem \ref{theorem: NFRC consequences} and item (1) in Proposition \ref{prop: NFRC basic properties}, it is also natural to ask the following question.
\begin{question}
Is the ISR property inherited by normal subgroups?
\end{question}

\begin{remark}
Note that $F_2\times F_2$ has the ISR property as shown in \cite{aj}. Proposition \ref{prop: NFRC basic properties} implies that $F_2\times F_2$ does not have non-factorizable regular character. In view of Theorem \ref{theorem: NFRC consequences}, we know that for non-amenable groups with trivial amenable radical, having non-factorizable regular character is a sufficient but not necessary condition for showing the ISR property.
\end{remark}

\begin{proposition} Let $H<G$ be countable discrete groups. Assume that $G$ is ICC and $[G:H]<\infty$. If $H$ has non-factorizable regular character, then $G$ also has non-factorizable regular character.
\end{proposition}
\begin{proof}
Let $\phi,\psi$ be two characters on $G$ such that $\phi(g)\psi(g)=0$ for all $e\neq g\in G$. We aim to show that $\phi\equiv\delta_e$ or $\psi\equiv\delta_e$.

Since $H$ has non-factorizable regular character and we may consider $\phi,\psi$ as characters on $H$, we deduce that $\phi|_H\equiv \delta_e$ or $\psi|_H\equiv \delta_e$. Without loss of generality, we assume that $\phi|_H\equiv \delta_e$. Let us show that in fact $\phi\equiv\delta_e$.

For any $e\neq g\in G$, denote by $C(g)$ the conjugacy class of $g$ in $G$. 
We observe that there exists some $s\in G$ such that $\sharp(C(g)\cap sH)=\infty$. To see this, just note that $C(g)=C(g)\cap G=C(g)\cap \sqcup_{s\in I}sH=\sqcup_{s\in I}(C(g)\cap sH)$, where $I$ denotes a collection of left $H$-cosets representatives in $G$. Since $\sharp I=[G: H]<\infty$ and $\sharp C(g)=\infty$ by the ICC assumption, we deduce that for some $s\in I$, $\sharp(C(g)\cap sH)=\infty$.

Write $C(g)\cap sH=\{g_n: n\geq 1\}$. Then note that $g_n^{-1}g_m\in H$ for all $n, m\geq 1$. Hence $\phi(g_n^{-1}g_m)=0$ for all $n\neq m$. Therefore, using Lemma \ref{lemma: vanishing character from many conjugates} we deduce that $\phi(g)=\phi(g_n)=0$. It follows that $\phi=\delta_e$.
\end{proof}

\begin{proposition}
    Let $G$ be a countable discrete group which is an increasing union of subgroups $G_i$ such that each $G_i$ has non-factorizable regular character. Then $G$ has non-factorizable regular character.
\end{proposition}
\begin{proof}
    Let $\phi$, $\psi$ be two characters on $G$ with $\phi(g)\psi(g)=0$ for all $e\neq g\in G$. Assume that $\phi\not\equiv \delta_e$ and $\psi\not\equiv \delta_e$, then there exists some $i$ such that $\phi|_{G_i}\not \equiv\delta_e\not \equiv \delta_e$. But this leads to a contradiction, since $G_i$ has non-factorizable regular character.
\end{proof}
%%%%%%%%%%%%%%%%%%%%%%%%%%%%%%%%%%%%%%%%%%%%%%%%%%%%%%%%%%%%%%%%%%%%%
\subsection{Non-factorizable regular character for groups arising from infinite simple groups}\label{subsection:NFRC_from_simple_groups} Let us first recall some notation and introduce some notations we will need in this section. 
For a group $G$ denote by $\IChar(G)$ the set of indecomposable characters on $G$. Recall that an invariant random subgroup ($\IRS$, for short) of $G$ is any Borel probability measure $\mu$ on the space $\Sub(G)$ of subgroups of $G$, which invariant with respect to the action of $G$ by conjugations $$g:\Sub(G)\to\Sub(G),\;g(H)=gHg^{-1},\;g\in G,\;H<G.$$ Denote by $\IRS(G)$ the space of IRS of $G$ and by $\IRS^e(G)$ the space of ergodic IRS of $G$. It is known that for any $\IRS$ $\mu$ of $G$ the formula
\begin{equation}\label{eq: chi_mu definition}
    \chi_{\mu}(g)=\mu(\{H\in Sub(G):~g\in H\}),\;g\in G,
\end{equation} defines a character on $G$.
\begin{proposition}\label{prop: examples of NSFC-simple groups}
The following classes of infinite ICC simple groups $G$ have non-factorizable regular character.
\begin{itemize}
\item[(1)] Character rigid groups, \ie groups $G$ for which $\IChar(G)=\{\delta_e,\textbf{1}_G\}$;
\item[(2)] $L(Alt)$-groups, \ie the groups which can be expressed as $G=\cup_{i\in\mathbb{N}}G_i$, where $G_i$ a strictly increasing chain of finite alternating groups.\cites{tt_algebra,thomas_etds}; 
\end{itemize}
\end{proposition}
\begin{proof}
(1) By definition, for any character $\phi$ on $G$, we may write $\phi=\lambda\delta_e+(1-\lambda)\textbf{1}_G$ for some $\lambda\in [0, 1]$. Thus, if $\phi\not\equiv \delta_e$, \ie $\lambda\neq 1$, then we have $\phi(s)=1-\lambda>0$ for all $e\neq s\in G$.
This clearly implies $G$ has non-factorizable regular character.

(2) We will split the proof  into two cases depending on whether $G\cong Alt(\mathbb{N})$.

Case 1: $G\cong Alt(\mathbb{N})$.

By Thoma's classification theorem \cite{thoma_mz} (see also \cite[\S~6]{thomas_etds}), we know that every indecomposable character on $G$ is the restriction of some indecomposable character on the overgroup $S_{\infty}$. More precisely, for any indecomposable character $\chi$ on $G$, there exist two sequences $\{\alpha_i\mid i\in\mathbb{N}^+\}$ and $\{\beta_i\mid~i\in\mathbb{N}^+\}$ of non-negative real numbers satisfying 
\begin{itemize}
\item $\alpha_1\geq \alpha_2\geq \cdots\geq\alpha_i\geq \cdots\geq 0$;
\item $\beta_1\geq\beta_2\geq\cdots\geq\beta_i\geq\cdots\geq 0$;
\item $\sum_{i=1}^{\infty}\alpha_i+\sum_{i=1}^{\infty}\beta_i\leq 1$;
\end{itemize}
and such that $\forall g\in G$, 
\[\chi(g)=\prod_{n=2}^{\infty}p_n^{k_n(g)},~\text{where}~p_n=\sum_{i=1}^{\infty}\alpha_i^n+(-1)^{n+1}\sum_{i=1}^{\infty}\beta_i^n.\]
In these products, $k_n(g)$ is the number of cycles of length $n$ in the cyclic decomposition of the permutation $g$ and $p_n^0$ is always taken to be 1, including the case $s_n=0$.

Set $s=(1~2~3)$, the 3-cycle. Then $k_n(s)=1$ if $n=3$ and is zero elsewhere. Hence $\chi(s)=p_3=\sum_{i=1}^{\infty}\alpha_i^3+\sum_{i=1}^{\infty}\beta_i^3\geq 0$ for all indecomposable $\chi$. Moreover, $\chi(s)=0$ iff $\alpha_i=0=\beta_i$ for all $i$, in which case, $\chi=\delta_e$. 

Now, let us check that for every character $\phi\neq\delta_e$ on $G$, $\phi(s)>0$. We may write $\phi=\int_{\IChar(G)}\chi d\mu(\chi)$ for some probability measure $\mu$ on $\IChar(G)$ (see, for example, \cite[Chapter IV, Theorem 8.21]{TakI}). Then $\phi(s)=\int_{\IChar(G)}\chi(s) d\mu(\chi)\geq 0$. If $\phi(s)=0$, then for $\mu$-a.e. $\chi$, we have $\chi(s)=0$. Hence $\chi=\delta_e$ for  $\mu$-a.e. $\chi$. Therefore, $\phi=\delta_e$. This  contradicts to our assumptions and implies that $G$ has non-factorizable regular character.

Case 2: $G\not\cong Alt(\mathbb{N})$. 

By \cite{thomas_etds}, we know that every indecomposable character $\chi\neq\delta_e$ arises from some ergodic IRS on $G$, \ie, there exists some $\mu\in \IRS^e(G)$ such that $\chi=\chi_{\mu}$ (see \eqref{eq: chi_mu definition}). In view of \cite[Definition 3.9]{tt_algebra}, we may split this case by considering whether $G$ is of linear or sublinear natural orbit growth. Moreover, by \cite[Theorem 3.7]{tt_algebra}, we may further assume that $G$ has almost diagonal type. 
Indeed, otherwise, $\IRS^e(G)=\{\delta_{\{e\}}, \delta_{\{G\}}\}$ and hence  $\IChar(G)=\{\chi_{\delta_{\{e\}}},\chi_{\delta_{\{G\}}}\}=
\{\delta_e,\textbf{1}_G\}$. Hence, $G$ is character rigid and the proof is finished by citing the previous item.

Subcase 1: $G$ has linear natural orbit growth.

In this subcase, Tucker-Drob and Thomas proved in \cite[Theorem 3.19]{tt_algebra} that   $IRS^e(G)=\{\delta_e,\delta_G,\nu_r:~r\geq 1\}$, where $G\curvearrowright (\Delta,m)$ is some  canonical  ergodic (in fact weakly mixing \cite[Theorem 2.14]{tt_algebra}) action  and $\nu_r$ is the associated stabilizer IRS for the diagonal action $G\curvearrowright (\Delta^{\otimes r}, m^{\otimes r})$. It is routine to check that $\chi_{\nu_r}(g)=m^r(\{(x_1,\ldots,x_r):~g\in \cap_{i=1}^rFix(x_i)\})=(m(\{x:~g\in Fix(x)\}))^r$ for all $g\in G$. Note that $\nu_r=\nu^r$ for all $r\geq 1$.

Since $\nu\not\in\{\delta_{\{e\}},\delta_{\{G\}}\}$ (as we assumed that $G$ has almost diagonal type), we may find some $e\neq s$ such that $0<\chi_{\nu}(s)<1$. Indeed, assume not, then $\chi_{\nu}=\delta_H$ for some normal subgroup $H\lhd G$. Since $G$ is simple, we deduce that $H=\{e\}$ or $H=G$. Recall that $\chi_{\nu}=m(\{x:~g\in Fix(x)\})=m(Fix(g))$. One can check that $\chi_{\nu}=\delta_{e}$ implies that $G\curvearrowright (\Delta, m)$ is essentially free, thus
$\nu=\delta_{\{e\}}$. In the case that $\chi_{\nu}=\delta_{G}$, the action $G\curvearrowright (\Delta,m)$ is trivial, hence $\nu=\delta_{\{G\}}$. In either case, we get a contradiction.

Next, we show that $\phi(s)>0$ for all character $\delta_e\not\equiv \phi$ on $G$.

Write $\phi=\lambda\delta_e+\mu \textbf{1}_G+(1-\lambda-\mu)\int_{r\geq 1}\chi_{\nu^r}d\theta(r)$ for some non-negative numbers $\lambda$, $\mu$ and some probability measure $\theta$ on $\{r\in\mathbb{N}:~r\geq 1\}$.
Without loss of generality, we may assume $\mu=0,\lambda<1$ and hence $1-\lambda-\mu>0$. Then, assume that $\phi(s)=0$, then $\chi_{\nu^r}(s)=0$ for $\mu$-a.e. $r$. In particular, $0=\chi_{\nu^r}(s)=\chi_{\nu}^r(s)$ for some $r\geq 1$. Thus, $\chi_{\nu}(s)=0$, which yields a contradiction to our choice of $s$. Therefore, we have shown that $\phi(s)>0$ if $\phi\not\equiv \delta_e$.

Subcase 2: $G$ has sublinear natural orbit growth.

By the proof of \cite[Theorem 3.21]{tt_algebra}, we may write $G=\cup_{\ell\in\mathbb{N}}G(\ell)$, where $\{G(\ell):~\ell\in\mathbb{N}\}$ is a non-decreasing $L(Alt)$-group with linear natural orbit growth. 

Fix any characters $\phi\neq\delta_e\neq\phi'$ on $G$, we may restrict them on subgroups $G(\ell)$ and hence get that there exists some $\ell\geq 1$ such that $\phi|_{G(\ell)}\neq \delta_e\neq \phi'|_{G(\ell)}$. 
So apply the previous proof to  this $G(\ell)$ to get some $s\in G(\ell)$ such that $\phi(s)>0$ and $\phi'(s)>0$. 
\end{proof}

\begin{remark}
Note that in all items except the subcase 2 in item 2, we actually proved a priori stronger property: 
there exists some non-trivial element $s\in G$ such that $\phi(s)>0$ for any character $\phi$ on $G$ with $\phi\not\equiv \delta_e$.
\end{remark}

\noindent For examples of character rigid groups, one may check \cite[Theorem 2.9]{dm_ggd2014}, \cite[Example 15.F.11]{bd_book} or any infinite groups with only two conjugacy classes as constructed in \cite{osin_annals}.

Motivated by the conditions used in \cite[Theorem 2.11]{dm_ggd2014}, we present more groups with non-factorizable regular characters built from simple groups. For concrete examples satisfying these conditions, see \cite[Lemma 3.5]{dm_ggd2014}.

\begin{proposition}\label{prop: examples of NFRC-relative ICC extension of simple groups}
 Let $G$ be a group and $R$ be an ICC subgroup of $G$ such that
 \begin{itemize}
 \item[(i)] $R$ has non-factorizable regular characters;
 \item[(ii)] for every $g\in G\setminus\{e\}$, there exists a sequence of distinct elements $\{g_i\}_{i\geq 1}\subset C_R(g)$ such that $g_i^{-1}g_j\in R$ for any $i, j$, where $C_R(g)=\{rgr^{-1}: r\in R\}$.
 \end{itemize}
Then $G$ has non-factorizable regular character.
 \end{proposition}
 \noindent Examples of groups $R$ satisfying the property (i) in Proposition \ref{prop: examples of NFRC-relative ICC extension of simple groups} are groups having no proper characters in the sense of \cite[Def. 2.4]{dm_ggd2014} (equivalently, character rigid groups).

\begin{proof}[Proof of Proposition \ref{prop: examples of NFRC-relative ICC extension of simple groups}]
Notice that $R$ is infinite by condition (ii) and the ICC assumption. In addition, condition (ii) implies that $R\subseteq G$ is relative ICC, in particular, $G$ is ICC and hence has trivial center.

Let $\phi$ and $\psi$ be any two characters on $G$ with $\phi(g)\psi(g)=0$ for all $e\neq g\in G$. We aim to show $\phi\equiv \delta_e$ or $\psi\equiv \delta_e$. By considering $\phi|_R$ and $\psi|_R$, we deduce that either $\phi|_R\equiv \delta_e$ or $\psi|_R\equiv\delta_e$ since $R$ has non-factorizable regular characters. Without loss of generality, we may assume that $\phi|_R\equiv \delta_e$.

Take any non-trivial element $g\in G$. From condition (ii) we deduce that $\phi(g_i^{-1}g_j)=0$ for all $g_i\in C_R(g)$ with $i\neq j$. Thus $\phi(g)=0$ by Lemma \ref{lemma: vanishing character from many conjugates}. Therefore, $\phi=\delta_e$, which finishes the proof.
\end{proof}

\subsection{Non-factorizable regular characters for charfinite groups}
Using this character approach, we may recover a special case for the result in \cite{kp}.

%{\color{blue} Could you please double check the proof of the following proposition? Since I am not that confident whether it is reasonable to introduce $\ICh_N(G)$ as in the proof, any measurability issue on this set?} {\color{green} Seems good to me. I have added some lines of explanations.} {\color{blue} I see, thank you for explanations.}

\begin{proposition}\label{prop: SLnZ}
Let $G=SL_n(\mathbb{Z})$ for odd $n\geq 3$. Then $G$ has non-factorizable regular character.
\end{proposition}
\begin{proof}
By \cite[Theorem 3]{bekka}, every indecomposable character on $G$ is exactly one of the following two types:

(i) either it is the indecomposable character of an irreducible finite dimensional representation of some congruence quotient $SL_n(\mathbb{Z}/N\mathbb{Z})$ for some $N\geq 1$, or

(ii) it is $\delta_e$.

\noindent
Recall that any character $\psi$ on $G$ can be written as an integral of indecomposable ones with respect to some Borel probability measure on $\IChar(G)$:
$$\psi(g)=\int_{\IChar(G)}\chi(g)d\mu_(\chi).$$ Let $\ICh_N(G)$ denote the set of all indecomposable characters obtained by an irreducible finite dimensional representation which factors through $SL_n(\mathbb{Z}/{N\mathbb{Z}})$, but not $SL_n(\mathbb{Z}/{N'\mathbb{Z}})$ for any $N'<N$. Observe that for any $N\in\mathbb N$ the subset $\ICh_N(G)\subset\IChar(G)\setminus\{\delta_e\}$ can be written as a countable combination of unions and differences of subsets of the form $\ICh_g(G)=\{\chi\in\IChar(G)\setminus\{\delta_e\}:\chi(g)=1\}.$ Therefore, $\ICh_N(G)$ is a measurable subset of $\IChar(G)$.

Now, let $\phi,\psi$ be any two characters on $SL_n(\mathbb{Z})$.
For any $g\in G$, we may write 
\begin{align*}
\phi(g)&=\sum_{N=1}^{\infty}c_N\int_{\ICh_N(G)}\chi(g)d\mu_N(\chi)+d\delta_e(g),\\
\psi(g)&=\sum_{N=1}^{\infty}c_N'\int_{\ICh_N(G)}
\chi(g)d\nu_N(\chi)+d'\delta_e(g).
\end{align*}
Here, $\mu_N,\nu_N$ are probability measures on $\ICh_N(G)$ with $\sum_{N=1}^{\infty}c_N+d=1=\sum_{N=1}^{\infty}c_N'+d'$ and $c_N, c_N', d, d'\geq 0$. Assume that $\phi\not\equiv \delta_e$ and $\psi\not\equiv \delta_e$, then $d<1$ and $d'<1$; equivalently, $c:=\sum_{N=1}^{\infty}c_N>0$ and $c':=\sum_{N=1}^{\infty}c_N'>0$.

Pick a sufficiently large $M$ such that $\sum_{N=1}^Mc_N>\frac{c}{2}$ and $\sum_{N=1}^Mc_N'>\frac{c'}{2}$. Denote by $\pi_N: SL_n(\mathbb{Z})\twoheadrightarrow SL_n(\mathbb{Z}/{N\mathbb{Z}})$ the natural epimorphism. Then note that $[G: Ker(\pi_N)]<\infty$ and $\cap_{N=1}^MKer(\pi_N)=Ker(\pi_{lcm(1,\ldots,N)})$. Hence, $[G:\cap_{N=1}^MKer(\pi_N)]<\infty$, in particular, $\cap_{N=1}^MKer(\pi_N)\neq \{e\}$.

Pick any $e\neq g\in \cap_{N=1}^MKer(\pi_N)$. Observe that $\chi(g)=1$ for all $\chi\in \ICh_N(G)$ and  all $1\leq N\leq M$. Thus, $\phi(g)=\sum_{N=1}^Mc_N+\sum_{N=M+1}^{\infty}c_N\int_{\ICh_N(G)}\chi(g)d\mu_N(g)$ and hence
\begin{align*}
|\phi(g)|&\geq \sum_{N=1}^Mc_N-\sum_{N=M+1}^{\infty}\int_{\ICh_N(G)}c_N|\chi(g)|d\mu_N(g)\\
&\geq \sum_{N=1}^Mc_N-\sum_{N=M+1}^{\infty}\int_{\ICh_N(G)}c_N1d\mu_N(g)\\
&=2\sum_{N=1}^Mc_N-c>0.
\end{align*}
Hence $\phi(g)\neq 0$. Similarly, $\psi(g)\neq 0$. Hence $\phi(g)\psi(g)\neq 0$.
\end{proof}

Let $\Gamma$ be a countable discrete group.
Denote by $Rad(\Gamma)$ the amenable radical of $\Gamma$, that is, the largest amenable normal subgroup. 
Let $\mathcal{P}(\Gamma)\subset \ell^{\infty}(\Gamma)$ be the convex set consisting of all normalized positive definite functions equipped with the weak$^*$-topology. Note that $\Gamma$ acts on $\mathcal{P}(\Gamma)$ by conjugation. Note that a fixed point of this conjugation action is precisely a character. Following \cite[Definition 1.1]{bbhp}, a character is called {\it amenable} if the corresponding GNS representation $(\pi, H)$ is amenable in the sense of Bekka \cite{Be89}, \ie $\pi\otimes \bar{\pi}$ weakly contains the trivial representation. It is called {\it von Neumann amenable} if moreover $\pi(\Gamma)''$ is an amenable von Neumann algebra. It is called {\it finite} if $H$ is finite dimensional.

Recall that in \cite[Definition 1.2]{bbhp}, the authors introduced the concept of charfinite groups, defined as follows.

\begin{definition}\label{def: charfinite groups}
A countable discrete group $\Gamma$ is called {\it charfinite} if it satisfies the following conditions:
\begin{enumerate}
    \item Every non-empty $\Gamma$-invariant compact convex subset $\mathcal{C}\subset\mathcal{P}(\Gamma)$ contains a fixed point with respect to the conjugation action, that is, a character.
    \item Every extremal character $\phi\in {\rm Char}(\Gamma)$ is either supported on $Rad(\Gamma)$ or von Neumann amenable.
    \item $Rad(\Gamma)$ is finite.
    \item $\Gamma$ has a finite number of isomorphism classes of unitary representations in each given finite dimension.
    \item Every amenable extremal character is finite.
\end{enumerate}
\end{definition}
\noindent Note that charfinite groups satisfy Margulis's normal subgroup theorem \cite[Proposition 3.3]{bbhp}, \ie any normal subgroup $N\lhd \Gamma$ is either finite or has finite index.
Examples of charfinite groups are established in \cites{bh1,bbhp,bbh}.

Although we would not need it in this paper, we make the following remark concerning charfinite groups.
\begin{remark}\label{remark: on KP's proof of ISR for modified charfinite groups}
    The proof of \cite[Theorem 1.1]{kp} actually shows that every \emph{modified charfinite group} with trivial amenable radical has 
the ISR property. Here, by a modified charfinite group, we mean a group $\Gamma$ satisfying all conditions of Definition \ref{def: charfinite groups} with the condition (2) replaced by the following stronger condition:

$(2')$ If $\pi: \Gamma\rightarrow U(N)$ is a projective representation such that $\pi(\Gamma)''$ is a finite factor, then $\pi(\Gamma)''$ is finite dimensional or $\pi$ extends to an isomorphism of $L(\Gamma)\cong \pi(\Gamma)''$, which implies $\pi(g)$ has trace zero for all $e\neq g$ in this case.

In other words, for a modified charfinite group we need condition $(2)$ to hold for all projective unitary representations instead of unitary representations. It is not clear whether this is equivalent to the original condition. While as commented in \cite{kp}, some known charfinite groups, e.g., \cite[Theorem C]{bh1} do satisfy this modified charfinite property. In fact, it might be possible to check that the proofs in the series of papers on establishing charfinite property in \cites{bh1, bbh, bbhp} still work for projective unitary representations.
\end{remark}

In view of Remark \ref{remark: on KP's proof of ISR for modified charfinite groups} and the fact that charfinite groups with trivial amenable radical are C$^*$-simple, it is natural to ask:
\begin{question}
Do all charfinite groups with trivial amenable radical have non-factorizable regular characters?
\end{question}
\noindent In particular, it is interesting whether it is possible to extend the proof showing that $SL_n(\mathbb{Z}), n\geq 3$, has non-factorizable regular character to other charfinite groups.

In fact, the proof of Proposition \ref{prop: SLnZ} relies on the following property of groups:
every finite dimensional unitary representation of $\Gamma$ has finite image and hence factors through a finite quotient. With this property at hand, we can repeat the proof there to deduce the following result.
\begin{proposition}\label{prop: examples of NFRC-charfinite groups_1}
Let $\Gamma$ be a charfinite group satisfying the following conditions:
\begin{enumerate}
    \item The amenable radical of $\Gamma$ is trivial: $Rad(\Gamma)=\{e\}$;
    \item every finite dimensional unitary representation of $\Gamma$ has finite image.
\end{enumerate}
Then $\Gamma$ has non-factorizable regular character.
\end{proposition}
\begin{proof}
Let $\phi,\psi$ be two characters on $\Gamma$. Assume that $\phi\not\equiv \delta_e$ and $\psi\not\equiv \delta_e$, we aim to show that there exists some $e\neq g\in \Gamma$ such that $\phi(g)\psi(g)\neq 0$.

In view of item (2) in the definition of charfinite groups and since $Rad(\Gamma)=\{e\}$, we may write that
\begin{align*}
    \phi=d\delta_e+\sum_{N\geq 1}c_N\int_{\ICh_N(\Gamma)} \chi(g)d\mu_N(\chi),\\
    \psi=d'\delta_e+\sum_{N\geq 1}c_N'\int_{\ICh_N(\Gamma)} \chi(g)d\nu_N(\chi).
\end{align*}
Here $\ICh_N(\Gamma)$ denotes the set of all indecomposable characters obtained by irreducible finite dimensional representations on $\mathbb{C}^N$, and $\mu_N,\nu_N$ are probability measures on $\ICh_N(\Gamma)$ with $\sum_{N=1}^{\infty}c_N+d=1=\sum_{N=1}^{\infty}c_N'+d'$ and $c_N, c_N', d, d'\geq 0$.

Since $\phi\not\equiv\delta_e$ and $\psi\not\equiv \delta_e$, we get that $d<1$ and $d'<1$. Equivalently, $c:=\sum_{N=1}^{\infty}c_N>0$ and $c':=\sum_{N=1}^{\infty}c_N'>0$.

Pick a sufficiently large $M$ such that $\sum_{N=1}^Mc_N>\frac{c}{2}$ and $\sum_{N=1}^Mc_N'>\frac{c'}{2}$. By our assumption, every unitary representation $\pi_{\chi}$ corresponding to $\chi\in \ICh_N(\Gamma)$ for all $1\leq N\leq M$ has finite image and by item (4) in the definition of charfinite groups (since isomorphic unitary representations have equal kernel and the intersection of finitely many finite index subgroups still has finite index), we deduce that $\Lambda:= \cap_{1\leq N\leq M}\cap_{\chi\in \ICh_N(\Gamma)}ker(\pi_{\chi})$ is a finite index non-trivial normal 
 subgroup in $\Gamma$. 

Fix any $e\neq g\in \Lambda$. Observe that $\chi(g)=1$ for all $\chi\in \ICh_N(\Gamma)$ for all $1\leq N\leq M$. Thus, $\phi(g)=\sum_{N=1}^Mc_N+\sum_{N=M+1}^{\infty}\int_{\ICh_N(\Gamma)}\chi(g)d\mu_N(g)$ and hence
\begin{align*}
    |\phi(g)|&\geq \sum_{N=1}^Mc_N-\sum_{N=M+1}^{\infty}c_N|\chi(g)|d\mu_N(g)\\
    &\geq \sum_{N=1}^Mc_N-\sum_{N=M+1}^{\infty}\int_{\ICh_N(\Gamma)}c_N1 d\mu_N(g)\\
    &=2\sum_{N=1}^Mc_N-c>0.
\end{align*}
Hence, $\phi(g)\neq 0$. Similarly, $\psi(g)\neq 0$.
\end{proof}

Concerning the condition on finite image in item (2) of Proposition \ref{prop: examples of NFRC-charfinite groups_1}, we make the following remarks.

1. Let $\Gamma$ be an infinite group $\Gamma$ with property (T) and let $\pi: \Gamma\rightarrow U(n)$ be a finite dimensional unitary representation. It may happen that $\pi(\Gamma)$ is infinite. Indeed, notice that  $SO(n, \mathbb{R})\subset U(n)$. Margulis \cite{mar1} and Sullivan \cite{sul} independently proved that $SO(n, \mathbb{R})$ contains countable infinite groups, with property (T) for $n\geq 5$. For instance, in Margulis's approach, one can take $\Gamma=SO(n, \mathbb{Z}(\frac{1}{5}))$, the subgroup of matrices in $SO(n, \mathbb{R})$ whose entries belong to $\mathbb{Z}(\frac{1}{5})$, the subring of $\mathbb{Q}$ generated by $\frac{1}{5}$. For other examples, see e.g., \cite[Example 1]{rap}.

2. To check item (4) in Definition \ref{def: charfinite groups} for various classes of groups, the authors in \cites{bh1, bbhp, bbh} used two approaches. If $\Gamma$ has property (T) then item (4) holds by Wang's classical result, e.g., \cite[Theorem 2.6]{wan}. Another approach is showing directly that every finite dimensional (irreducible) unitary representation of $\Gamma$ has finite image, see e.g., the proof of \cite[Proposition 7.1]{bbhp}.

3. In \cite[Proposition 4.C.14]{bd_book}, it was shown that for a discrete group $\Gamma$, the image of every finite dimensional unitary representation is finite iff the continuous surjective homomorphism $c_P^B: \text{Bohr}(\Gamma)\twoheadrightarrow \text{Prof}(\Gamma)$ is an isomorphism, where $\text{Bohr}(\Gamma)$ denotes the Bohr compactification, see \cite[Theorem 4.C.3]{bd_book} and $\text{Prof}(\Gamma)$ denotes the profinite completion, see \cite[Theorem 4.C.11]{bd_book}. It seems that this property is very rare. 

Below, we present another more direct approach by using \cite[Proposition 14]{bekka-23}, which we recall below.

\begin{proposition}\label{prop: distorted elelment maps to torsion}
Let $\Gamma$ be a finitely generated group and $\gamma\in\Gamma$ be a distorted element. Then, for every finite dimensional unitary representation $\pi: \Gamma\rightarrow U(N)$ of $\Gamma$, the matrix $\pi(\gamma)\in U(N)$ has finite order.
\end{proposition}
Here, in a finitely generated group $\Gamma$ with a finite generating set $S$, an element $\gamma\in\Gamma$ is called distorted if the translation number $t(\gamma)=0$, where 
\begin{align*}
    t(\gamma)=\lim\inf_{n\to\infty}\frac{\ell_S(\gamma^n)}{n}.
\end{align*}

Clearly, we have
\begin{proposition}\label{prop: examples of NFRC-charfinite groups_2}
Let $\Gamma$ be a finitely generated charfinite group satisfying the conditions:
\begin{enumerate}
    \item $Rad(\Gamma)=\{e\}$;
    \item $\Gamma$ contains a distorted element with infinite order. 
\end{enumerate}
Then $\Gamma$ has non-factorizable regular character.
\end{proposition}
\begin{proof}

Let $\phi,\psi$ be two characters on $\Gamma$. Assume that $\phi\not\equiv \delta_e$ and $\psi\not\equiv \delta_e$, we aim to show there exists some $e\neq g\in \Gamma$ such that $\phi(g)\psi(g)\neq 0$.

Following the same proof as before, in view of Proposition \ref{prop: distorted elelment maps to torsion}, 
it suffices to observe that we may take $g$ be to $\gamma^n$ for high enough power $n$, where $\gamma$ is a distorted element with infinite order in $\Gamma$.
\end{proof}

On the one hand, note that irreducible non-uniform lattices in higher rank Lie groups contain distorted elements with infinite order as shown in \cite{lmr}. It is clear to see that property (U3) (and hence also (U1), (U2)) in \cite{lmr} implies being distorted. On the other hand, \cite[Corollary 2.6, 2.9]{lmr} show that both uniform lattices in a semisimple Lie group and hyperbolic groups do not contain a U-element.

\subsection{Examples without non-factorizable regular characters}

Via Pontryagin duality, we observe that $\mathbb{Z}$ does not have non-factorizable regular character.
We are grateful to Huichi Huang for  helpful discussion on \cite{rv}.
\begin{proposition}
The integer group $\mathbb{Z}$ does not have non-factorizable regular character.
\end{proposition}
\begin{proof}
By Fourier transform, we have a bijection between characters on $\mathbb{Z}$ and measures on the unit circle $\mathbb{T}$. More precisely,
for any character $\phi$ on $\mathbb{Z}$, there is a unique measure $\mu_{\phi}$ on the unit circle $\mathbb{T}$ such that $\int_{\mathbb{T}}z^nd\mu_{\phi}(z)=\phi(n)$ for all $n\in\mathbb{Z}$. Conversely, given any measure $\mu$ on $\mathbb{T}$, the map $\mathbb{Z}\ni n\mapsto \int_{\mathbb{T}}z^nd\mu(z)\in\mathbb{\mathbb{C}}$ defines a character on $\mathbb{Z}$. Hence, for two characters $\phi,\psi$ on $\mathbb{Z}$, we have $\phi(n)\psi(n)=\int_{\mathbb{T}}z^nd(\mu_{\phi}*\mu_{\psi})(z)$ for all $n\in\mathbb{Z}$, where $\mu_{\phi}*\mu_{\psi}$ denotes the convolution of two measures.  Moreover, for a measure $\theta$ on $\mathbb{T}$, we know that $\theta=\text{Haar}$, the Haar measure on $\mathbb{T}$, iff $\int_{\mathbb{T}}z^nd\theta(z)=0$ for all $n\neq 0$.

Hence, to show $\mathbb{Z}$ does not have non-factorizable regular character, we just need to construct two non-Haar measures $\mu,\nu$ on $\mathbb{T}$ such that $\mu*\nu=\text{Haar}$. Recall that in \cite[Theorem 3.1]{rv}, Rao and Varadarajan constructed two singular measures $\mu$ and $\nu$ (with respect to the Haar measure) on the unit circle $\mathbb{T}$
such that the convolution $\mu*\nu=\text{Haar}$, the Haar measure on $\mathbb{T}$. 
Thus we are done.
\end{proof}

\iffalse
\begin{remark}\label{remark: property OC can be defined using all characters}
    Note that the above proof actually shows that when defining non-factorizable regular character, we can relax the assumption that $\phi,\psi$ takes non-negative real values. In other words, we have $G$ has non-factorizable regular character iff for any two characters $\phi,\psi$ on $G$, $\phi(g)\psi(g)=0$ for all $e\neq g\in G$ implies $\phi\equiv \delta_e$ or $\psi\equiv\delta_e$. Indeed, suppose $G$ has non-factorizable regular character and there are two characters $\phi,\psi$ on $G$ with $\phi(g)\psi(g)=0$ for all $e\neq g\in G$,  then set $\phi'(g)=\vert\phi(g)\vert^2=\phi(g)\overline{\phi(g)}$ and similarly $\psi'(g)=\vert \psi(g)\vert^2$. Since $g\mapsto \overline{\phi(g)}$ is still a character on $G$ and it is not hard to see $\phi'$ is still a character on $G$ since the Schur product of two positive matrices is still positive. Hence from the assumption $\phi'(g)\psi'(g)=0$ for all $e\neq g\in G$ and non-factorizable regular character, we deduce that $\phi'(g)\equiv \delta_e$ or $\psi'(g)\equiv \delta_e$. Clearly, this is equivalent to $\phi(g)\equiv \delta_e$ or $\psi(g)\equiv \delta_e$.
\end{remark}
\fi

%For finite cyclic groups, maybe order $n$ cyclic groups do not have non-factorizable regular character for all $n\geq 3$. Hence all abelian groups containing $n$-torsion elements for any $n\geq 3$ do not have non-factorizable regular character.

\subsection{Questions on the non-factorizable regular character property}\label{subsection: questions on non-factorizable regular character property}

We do not know whether every ICC acylindrical hyperbolic group with trivial amenable radical have non-factorizable regular character. If they do, then we would get a new proof of the result in \cite{cds}.
However, even for free groups $F_n$, $n\geq 2$, it is not clear how to check this property. In view of Proposition \ref{prop: NFRC basic properties}, we observe that non-abelian free groups do not contain non-trivial normal subgroups with trivial intersection. Indeed, assume there are two non-trivial subgroups $H, K\lhd F_n$ with $H\cap K=\{e\}$. Then $H$ commutes with $K$. Pick any $e\neq h\in H$ and $e\neq k\in K$, we have $hk=kh$. In free groups, this implies that
$h=t^l, k=t^m$ for some $t\in F_n$ and $l, m\in\mathbb{Z}\setminus\{0\}$. Therefore, $e\neq t^{lm}\in H\cap K$, yielding the desired contradiction.

\begin{question}
Does the non-abelian free group $F_n$ have the non-factorizable regular character?
In particular, assume that $\phi(g)=\mu\{H\leq F_n: g\in H\}$ and $\psi(g)=\nu\{H\leq F_n: g\in H\}$, where $\mu,\nu\in \IRS(F_n)$ (see Subsection \ref{subsection:NFRC_from_simple_groups}).
%, the set of invariant random subgroups on $F_n$, i.e., $F_n$-conjugate invariant probability measures on the space of subgroups of $F_n$.
Does it hold that $\phi(g)\psi(g)=0$ for all $e\neq g\in F_n$ implies $\phi\equiv \delta_e$ or $\psi\equiv \delta_e$?
\end{question}

The following proposition was communicated to us by Prof. Hanfeng Li. Note that it gives an affirmative answer to the second half of the above question. We thank him for allowing us to include it here.

\begin{proposition}[Hanfeng Li]\label{prop: hanfeng's proof}
Let $G=F_2$, the non-abelian free group on two generators. Let $\phi(g)=\mu(Fix(g))$ and $\psi(g)=\nu(Fix(g))$, where $G\curvearrowright (X,\mu)$ and $G\curvearrowright (Y,\nu)$ are two p.m.p. actions and $Fix(g)$ denotes the fixed points set of $g$. Suppose that $\phi(g)\psi(g)=0$ for all $e\neq g\in G$, equivalently, the diagonal action $G\curvearrowright (X\times Y,\mu\times \nu)$ is essentially free. Then either $\phi\equiv \delta_e$ or $\psi\equiv \delta_e$; equivalently, at least one of the two actions is essentially free.
\end{proposition}
\begin{proof}
Assume not, then let $e\neq s, t$ be two non-trivial elements in $G$ such that $\mu(Fix(s))>0$ and $\nu(Fix(t))>0$.

Pick an $n\geq 1$ such that $W:=Fix(s)$ satisfies that $\mu(W)>\frac{1}{n}$. Then since $\sum_{i=0}^n\mu(t^iW)>1$, we deduce that for some $0\leq i\neq j\leq n$ we have $\mu(t^iW\cap t^jW)>0$. In particular, there is some $j_0>0$ such that $\mu(W\cap t^{-j_0}W)>0$.
This implies that for some non-null set $V\subset W\subset X$ we have $\mu(V)>0$ and $t^{j_0}V\subset W$.

Since $V\subset W\subset Fix(s)$ and $t^{j_0}V\subset W$, we deduce that $V\subset Stab(t^{-j_0}st^{j_0})$ and we also have $V\subset Stab(s^{-1}t^{-j_0}st^{j_0})$. Observe that we may replace $s$ by any $s^i$ for all $i\in\mathbb{Z}$ and hence deduce that 
$$V\subset Fix(s^{-i}t^{-j_0}s^it^{j_0}),~\forall~ i\in\mathbb{Z}.$$
By symmetry, \ie by repeating the above argument for the 2nd action, we deduce there is some $i_0\neq 0$ and some non-null set $V'\subset Y$ such that
$$V'\subset Fix(s^{-i_0}t^{-j}s^{i_0}t^j),~\forall~j\in\mathbb{Z}.$$
This implies that
$(\mu\times \nu)(Fix(s^{-i_0}t^{-j_0}s^{i_0}t^{j_0}))\geq (\mu\times\nu)(V\times V')=\mu(V)\nu(V')>0.$

If $s^{-i_0}t^{-j_0}s^{i_0}t^{j_0}\neq e$, then we get a contradiction to the essential freeness of the diagonal action. 

If $s^{-i_0}t^{-j_0}s^{i_0}t^{j_0}=e$, then $s^{i_0}t^{j_0}=t^{j_0}s^{i_0}$. Thus since every subgroup in $G$ is free and $i_0j_0\neq 0$, we deduce that $\langle s^{i_0}, t^{j_0}\rangle\cong \mathbb{Z}=\langle h\rangle$. Hence $s^{i_0}=h^m$ and $t^{j_0}=h^{m'}$ for some non-zero $m,m'\in\mathbb{Z}$. Then $e\neq g:=h^{mm'}\in\langle s\rangle\cap \langle t\rangle$ has  non-null fixed points sets (since $Fix(s)\cup Fix(t)\subseteq Fix(g)$) and thus we still get a contradiction to the essential freeness of the diagonal action.
\end{proof}

\iffalse
Concerning the above question (with $\phi,\psi$ coming from IRSs), I have tried to:

(1) use the theorem in \cite{blt} saying that free groups are P-stable and hence every IRS is cosofic to approximate $\mu,\nu$ by finitely supported atomic measures.

(2) use the support of IRS as studied in \cite{ht}.

But I still have not succeeded.
\fi

 We record one simple observations on checking whether free groups have non-factorizable regular characters. The following proposition shows that to check non-factorizable character property for $F_2$ it is sufficient to restrict the attention to characters with trivial kernel.
\begin{proposition}
Assume that for any characters $\phi,\psi$ on $F_2$, such that $\phi\psi=\delta_e$ and $\ker(\phi)=\ker(\psi)=\{e\}$, one has either $\phi=\delta_e$ or $\psi=\delta_e$. Then $F_2$ has the non-factorizable regular character.
%To check that $F_2$ has non-factorizable regular character or not, we may start with two characters $\phi,\psi$ on $F_2$ with $\phi(g)\psi(g)=0$ for all $e\neq g\in F_2$ and further assume that $ker(\phi)=ker(\psi)=\{e\}$.
\end{proposition}
\begin{proof} Let $\phi,\psi$ be two characters on $F_2$ and $\phi\psi=\delta_e$.
Assume that $H:=\ker(\phi):=\{g\in F_2: \phi(g)=1\}\neq \{e\}$.
Note that $H\lhd F_2$ and hence $H$ is non-amenable and isomorphism to $F_n$. For any $e\neq g\in F_2$, the centralizer $C(g)$ is cyclic and hence $[H: H\cap C(g)]=\infty$. Thus there is an infinite sequence $g_n\in H$ such that $g_nC(g)\cap g_mC(g)=\emptyset$ for all $n\neq m$.

Define $s_n=g_ngg_n^{-1}$. Note that the choice of $g_n$ guarantees that $s_n\neq s_m$ for all $n\neq m$. We have $s_n^{-1}s_m=g_ng^{-1}g_n^{-1}g_mgg_m^{-1}\in g_ng^{-1}Hgg_m^{-1}=H$. Hence $\psi(s_n^{-1}s_m)=0$ since $\phi|_H=1$ implies $\psi|_H=0$.
Therefore, $\psi(s_n)=\psi(g)=0$ by Lemma \ref{lemma: vanishing character from many conjugates}. Hence $\psi=\delta_e$.
\end{proof}

\section{Approximately finite groups}\label{sec:approximately finite}
In this section we show that full groups of simple Bratteli diagrams (known also as approximately finite groups) have the ISR property and prove some related useful statements. The proof relies on several auxiliary lemmas and will be given at the end of the section.

\subsection{Preliminaries on Bratteli diagrams and corresponding AF groups.}
The AF full groups arose from the study of orbit equivalence theory of Cantor
minimal systems developed in a series of papers by Herman–Putnam–Skau \cite{HermanPutnamSkau: 1992} and
Giordano-Putnam-Skau \cite{GPS: 1999} and were motivated by applications to the theory of $C^*$-algebras. Giordano-Putnam-Skau \cite{GPS: 1995} showed
that Cantor minimal $\mathbb Z$-systems are strong orbit equivalent if and only if the associated
AF full groups are isomorphic as abstract groups and that the isomorphism of
crossed product $C^*$-algebras is completely characterized by the strong orbit equivalence
of underlying dynamical systems. We notice that the group $G_B$ is not always simple, however the commutator subgroup $G_B'$ is always simple.

In \cite{dm_jfa2013} the first named author and Medynets described the characters of the full groups $G_B$ of Bratteli diagrams $B$ for the case when the diagram $B$ is simple and $G_B$ is simple and admits only finitely many ergodic invariant measures. In a manuscript in preparation \cite{dm_manuscript} they describe the characters on each of the groups $G_B$, $G_B'$  with the assumption of simplicity of $B$ only, and thus in the cases when the set of invariant ergodic measures can be countably or even uncountably infinite. One of the main results can be simplified as follows.
\begin{theorem}\label{theorem: characters AF groups} Let $G$ be one of the groups $G_B,G_B'$. Then \begin{itemize}
\item[$1)$] for any non-trivial and non-regular indecomposable character $\chi$  on $G$ there exist a measure-preserving action of $G$ on a Borel probability space $(X,\mu)$ and a homomorphism $\rho:G/G_B'\to S^1=\{z\in\mathbb C:|z|=1\}$ such that \begin{equation}\label{EqChiMainTh}\chi(g)=\rho([g])\mu(\fix(g))\text{ for all }g\in G.\end{equation}
\item[$2)$] every character defined according to $1)$ is indecomposable.
\end{itemize}
\end{theorem}
\noindent In particular, any non-regular character $\chi$ on $G_B'$ has positive values on each element $g\in G_B'$ with $\suppg(g)\neq X$. This implies that $G_B$ and $G_B'$ have non-factorizable regular character. However,  we do not use either latter statements or Theorem \ref{theorem: characters AF groups} in the proofs of the results of the present section.

Now, let us recall some basic definitions and properties concerning Bratteli diagrams. For details we refer the reader to the papers of Herman-Putnam-Skau
\cite{HermanPutnamSkau:1992}, Dudko-Medynets \cite{dm_jfa2013} and
Bezuglyi-Kwiaktowski-Medynets-Solomyak
\cite{BezuglyiKwiatkowskiMedynetsSolomyak:2013}. 
\begin{definition}\label{Definition_Bratteli_Diagram} A {\it Bratteli diagram} is
an infinite graph $B=(V,E)$ such that the vertex set
$V=\bigcup_{i\geq 0}V_i$ and the edge set $E=\bigcup_{i\geq 1}E_i$
are partitioned into disjoint subsets $V_i$ and $E_i$ such that

(i) $V_0=\{v_0\}$ is a single point;

(ii) $V_i$ and $E_i$ are finite sets;

(iii) there exist a range map $r$ and a source map $s$ from $E$ to
$V$ such that $r(E_i)= V_i$, $s(E_i)= V_{i-1}$, and
$s^{-1}(v)\neq\emptyset$, $r^{-1}(v')\neq\emptyset$ for all $v\in V$
and $v'\in V\setminus V_0$.
\end{definition}
\noindent The pair $(V_i,E_i)$ or just $V_i$ is called the $i$-th level of the diagram $B$.
 A finite or infinite sequence of edges $(e_i : e_i\in E_i)$ such
that $r(e_{i})=s(e_{i+1})$ is called a {\it finite} or {\it infinite
path}, respectively. We write $e(v,v')$ to denote a path $e$ such
that $s(e)=v$ and $r(e)=v'$. For a Bratteli diagram $B$, we denote
by $X_B$ the set of infinite paths starting at the vertex $v_0$:
$$X_B = \left\{\{e_n\}\in \prod_{i\geq 1}E_i : s(e_{i+1}) = r(e_i)\mbox{ for every }i\geq 1\right\}.$$ Then $X_B$ is a
0-dimensional compact metric space with respect to the product topology.

Given a Bratteli diagram $B$, for every $n\geq 1$ denote by $G_n$ the group of homeomorphisms of $X_B$ that permute only the initial $n$ segments of the infinite paths $\{e_1,\ldots,e_n,e_{n+1}\ldots\}\in X_B.$  For each $n\geq 1$ and each vertex $v\in V_n$, denote by $X_v^{(n)}$ the set of all paths $\{e_1,e_2,\ldots\}$ such that $r(e_n) = v$. Denote by $h_v^{(n)}$ the number of finite paths connecting the root vertex $v_0$ to the vertex $v$ and denote by $E(v_0,v)$ the set of finite paths connecting these vertices.  let $X_v^{(n)}(\overline e)$, $\overline e\in E(v_0,v)$, be the set of infinite paths whose first $n$ initial segments coincide with those of $\overline e$. Note that $X_v^{(n)}(\overline e)$ is a clopen set.
Then for every $n\geq 1$ and every $v\in V_n$, we have that $$X_B =  \bigsqcup_{w\in V_n}X_w^{(n)}  \mbox{ and }X_v^{(n)} = \bigsqcup_{\overline e \in E(v_0,v)}X_v^{(n)}(\overline e). $$

Denote by $G_v^{(n)}$ the subgroup of $G_n$ whose elements permute only the first $n$ segments of paths from $X_v^{(n)}$. Thus,  the group $G_v^{(n)}$ is isomorphic to the symmetric group on $\{X_v^{(n)}(\overline e) : \overline e\in E(v_0,v)\}$.  It follows that $$G_n = \prod_{v\in V_n} G_v^{(n)}.$$  Set $G_B = \bigcup_{n\geq 1}G_n$. Note that $G_n\subset G_{n+1}$ for every $n\geq 1$ and $G_B$ is a locally finite group.

\begin{definition}
  \label{DefinitionAFFullGroup} Given a Bratteli diagram $B$, the group $G_B$ defined above is called the {\it full group associated to the diagram $B$}. We will simply write $G$ when the diagram $B$ is obvious form the context.
\end{definition}

The following remark reveals the connection between algebraic properties of the full groups and combinatorial properties of the associated Bratteli diagrams. The proofs and the references to the proofs can be found in Dudko-Medynets \cite[Section 2.1]{dm_jfa2013}.

\begin{remark}\label{remark: full group} Let $B = (V,E)$ be a Bratteli diagram and $G_B$ be the associated full group.

\begin{enumerate}

\item The dynamical system $(X_B,G_B)$ is {\it minimal}, that is every $G_B$-orbit is dense, if and only if the Bratteli diagram $B$ is {\it simple}, that is, for every $n\geq 1$ there exists $m> n$ such that every vertex in $V_n$ is connected to every vertex in $V_m$.

\item The dynamical system $(X_B,G_B)$ is {\it minimal} if and only if the commutator subgroup of $G_B$ is simple, \ie has no non-trivial normal subgroups.

\item The group $G_B$ is simple if and only if for every $n\geq 1$ there exists $m> n$ such that every vertex in $V_n$ is connected to every vertex in $V_m$ and the number of paths between these vertices is even. We refer to Bratteli diagrams with this property as {\it even diagrams}. In this case, the group $G_B$ coincides with its commutator subgroup.
\end{enumerate}
\end{remark}
The main result of this section is:
\begin{theorem}\label{thm: approximately finite groups have the ISR property}
Let $B$ be a simple Bratteli diagram and $G_B$ be the associated full group. Then $G_B$ has the ISR property.
\end{theorem}

\subsection{Auxiliary lemmas.}
For any countable group $G$ and a set of elements $S\subset G$ let $C_G(S)$ be its centralizer. Given an element $g\in G$ let us call an element $h\in G$ \emph{finitely permuted by the centralizer} of $g$ if $\sharp\{t^{-1}ht\mid t\in C_G(\{g\})\}<\infty$. Denote the set of such elements by $\fpc(g)$. The following statement seems well known:
 \begin{lemma} For any countable group $G$ and any $g\in G$ one has: $E(g)\in L(\fpc(g))$.
 \end{lemma}
 \begin{proof}
 For any $t\in C_G(\{g\})$, we have $tE(g)t^{-1}=E(tgt^{-1})=E(g)$. Thus by Proposition \ref{prop: basic properties of c.e. related to the ISR property} (2) and Proposition \ref{prop: control supports for elements in relative commutants}, we deduce that \[E(g)\in L(C_{G}(\{g\}))'\cap L(G)\subseteq L(\fpc(g)).\] 
 \end{proof}
\begin{lemma}\label{lemma: fpc preserves orbits} Let $G$ be a countable group acting faithfully by homeomorphisms on a Hausdorff topological space $X$. Assume that for any open set $A\subset X$ there exists $t\in G$ such that $\varnothing\neq\suppg(t)\subset A$. Let $g,h$ be such that $g$ is of finite order ($g^K=e$) and $h$ does not preserve orbits of $g$, \ie there exists $y\in X$ with $hy\notin \{g^ly:l\in\mathbb Z\}$. Then $h\notin \fpc(g)$.
\end{lemma}
\begin{proof} By continuity of $g,h$ there exists a neighborhood $U$ of $y$ such that the sets $U$, $g(U)$, $\ldots$, $g^{K-1}(U)$ do not intersect $h(U)$. Let $k\leqslant K$ be the maximum of periods of points inside $U$ under the action of $g$ and let $z\in U$ be of period $k$.  By continuity of $g$, there exists a neighborhood $A$ of $z$ such that $g^k$ acts trivially on $A$. Making $A$ smaller if necessary we may achieve that the sets $g^l(A)$, $l=0,\ldots k-1$, are pairwise disjoint. Notice that they do not intersect $h(A)$, since $A\subset U$.

Further, if $X$ has an isolated point $x$, considering the open set $A=\{x\}$ we obtain that there is no homeomorphism $t$ of $X$ with $\varnothing\neq\suppg(t)\subset A$. This contradicts to the conditions of the lemma. Thus, $X$ has no isolated points. Using the latter and that $X$ is a Hausdorff space one can construct a sequence of pairwise disjoint open sets $U_j\subset A\setminus \{z\}$ convergent to $z$. By the conditions of the lemma, there exists a sequence of elements $s_j$ with $\varnothing\neq\suppg(s_j)\subset U_j$. Observe that for any $j$ the elements $g^ls_jg^{-l}$, $l=0,\ldots,k-1$, are pairwise commuting. Set $$t_j=\prod\limits_{l=0}^{k-1}(g^ls_jg^{-l}).$$ By construction, $t_j$ commutes with $g$ for each $j$ and $\suppg(t_j)$ are pairwise disjoint. Next, for each $j$ fix a point $y_j\in \suppg(s_j)$. Then one has: $$t_j^{-1} ht_j(y_j)=ht_j(y_j),\;\;t_l^{-1} ht_l(y_j)=h(y_j)\;\;\text{for}\;\;l\neq j.$$ Since $t_jy_j\neq y_j$ we obtain that the elements $t_j^{-1} ht_j$ are pairwise disjoint, which finishes the proof.
\end{proof}
Now, let $G=G_B$ be a full group of a simple Bratteli diagram $B=(V,E)$ (see Remark \ref{remark: full group}) and let $X=X_B$ be the path space of $B$. % Let us follow the notations from Section 2 of \cite{dm_jfa2013}. 
Recall that $G$ is an increasing union of finite groups $G_n$. % (see the paragraph after Definition 2.6 in \cite{dm_jfa2013}). Notice that the groups $G_n$ play the role of $S(2^n)$ from the case $G=S(2^\infty)$. In particular, if the diagram $B$ have one vertex and two edges at each level then $G_n$ is isomorphic to $S(2^n)$ for each $n$ and $G$ is isomorphic to $S(2^\infty)$.
For $n\in\mathbb N$, let $G_{n,\infty}$ be the subgroup of $G$ consisting of elements $g\in G$ such that for any infinite path $(e_1,e_2,\ldots)\in X$ one has $g(e_1,\ldots e_n,e_{n+1},\ldots)=(e_1,\ldots,e_n,e'_{n+1},e'_{n+2},\ldots)$, where $(e'_{n+1},e'_{n+2},\ldots)$ depends only on $(e_{n+1},e_{n+2},\ldots)$. Observe that $G_{n,\infty}$ commutes with $G_n$. In fact, $G_{n,\infty}$ is the centralizer of $G_n$, but we would not use this statement. %The subgroup $G_{n,\infty}$ plays the role of $S_n(2^\infty)$. 
Following the notation used in \cite{dm_jfa2013}, given $j<n$ and a path $(e_{j+1},e_{j+2},\ldots,e_n)$ joining a vertex from $V_j$ with a vertex from $V_n$, where $e_i\in E_i,j+1\leqslant i\leqslant n$, we write $$U(e_{j+1},e_{j+2},\ldots, e_n)=\{x\in X: x_i=e_i, i=j+1,\ldots,n\}.$$  % Similarly to Claim on page 1 one can show
 \begin{lemma}\label{lemma: fpc G_n} For any $h\notin G_n$ one has: $\sharp\{t^{-1}ht\mid t\in G_{n,\infty}\}=\infty$.
 \end{lemma}

\begin{proof}
 Let $n\in\mathbb N$ and $h\notin G_n$. From the structure of the groups $G_n$ one can deduce that there are two (non-exclusive) possibilities.
 \vskip 0.2cm\noindent ${\bf a)}$ There exists $x=(x_1,x_2,\ldots)\in X_B$ with $hx=y=(y_1,y_2,\ldots)\in X_B$ such that $(y_{n+1},y_{n+2},\ldots)\neq (x_{n+1},x_{n+2},\ldots)$. Let $m$ be the minimal number greater than $n$ such that $x_m\neq y_m$. Let $v=r(x_m)\in V_m$ be the end vertex of the edge $x_m$.  The edges and vertices of all infinite paths $(e_{m+1},e_{m+2},\ldots)$ starting at the vertex $v$  form some Bratteli diagram $B_v$. One can show that $B_v$ is simple, given that $B$ is simple, and so the full group $\widetilde G=G_{B_v}$ of $B_{v}$ acts minimally on the space $\widetilde X=X_{B_v}$. Set $\widetilde x=(x_{m+1},x_{m+2},\ldots)$. Then there exists a sequence of elements $\widetilde t_m\in \widetilde G$ such that the points $\widetilde t_m\widetilde x\in \widetilde X$ are pairwise distinct. Introduce the elements $t_k\in G_B$ by:
 \begin{equation}t_k(z)=\left\{\begin{array}{ll}(z_1,\ldots,z_m,\widetilde t_k(z_{m+1},z_{m+2},\ldots),&\text{if}\;\;z_m=x_m,\\ z,&\text{otherwise}.
 \end{array}\right.
 \end{equation}
 Then $t_k\in G_{n,\infty}$ for every $k$, since its action depends only on the edges starting from $m$th with $m>n$. By construction, $t_k(y)=y$ for all $k$ and the points $t_k(x)$ are pairwise distinct.  It follows that for every $k$ one has: $t_kht_k^{-1}(t_k(x))=y$, from which we conclude that the elements $t_kht_k^{-1}$ are pairwise distinct.
 \vskip 0.2cm\noindent
 ${\bf b)}$ There exists $x^{(1)},x^{(2)}\in X_B$ starting with the same prefix of length $n$, \ie $x^{(1)}_j=x^{(2)}_j$ for $1\leqslant j\leqslant n$, but such that $y^{(i)}=hx^{(i)},i=1,2,$
 have different prefixes of length $n$, \ie $y^{(1)}_j\neq y^{(2)}_j$ for some $1\leqslant j\leqslant n$. Without loss of generality we may assume that $a)$ does not hold. In particular, $y^{(i)}_j=x^{(i)}_j$ for $i=1,2$ and all $j\geqslant n+1$. Fix $N>n$ such that $h\in G_N$. Without loss of generality we may assume that  $x^{(1)}_j=x^{(2)}_j$ for all $j\geqslant N+1$. Let $s\in G$ be the transposition of the finite paths $(x^{(1)}_{n+1},\ldots,x_N^{(1)})$ and $(x^{(2)}_{n+2},\ldots,x_N^{(2)})$, \ie $$s(x_1,\ldots, x_n,x_{n+1}^{(i)},\ldots,x_{N}^{(i)},x_{N+1},x_{N+2},\ldots)=(x_1,\ldots, x_n,x_{n+1}^{(3-i)},\ldots,x_{N}^{(3-i)},x_{N+1},x_{N+2},\ldots),$$ $i=1,2,$ for any $x_1,x_2,\ldots,x_n,x_{N+1},x_{N+2},\ldots$, and $s$ acts trivially on all other points of $X_B$.

 %By induction, we can construct an increasing sequence of numbers $N_j>N$ and a sequence of elements $t_j\in G_{n,\infty}$ of order 2 such that \begin{itemize}
  %   \item{} 
 %$\suppg(t_j)\subset U((x^{(1)}_{n+1},x^{(1)}_{n+2},\ldots,x^{(1)}_{N_j}))\cup t_j(U((x^{(1)}_{n+1},x^{(1)}_{n+2},\ldots,x^{(1)}_{N_j})))$ for each $j$, \item{} $t_j(U((x^{(1)}_{n+1},x^{(1)}_{n+2},\ldots,x^{(1)}_{N_j})))\subset U((x^{(2)}_{n+1},x^{(2)}_{n+2},\ldots,x^{(2)}_{N}))$ for each $j$,
 %\item{}the subsets $t_j(U((x^{(1)}_{n+1},x^{(1)}_{n+2},\ldots,x^{(1)}_{N_j})))$ are pairwise disjoint for $j\in\mathbb N$.\end{itemize} Then $t_j^{-1}ht_j(x^{(1)})$ are pairwise disjoint.

 Pick a disjoint sequence $\{U_j\}_{j\in\mathbb N}$ of non-empty $G_n$-invariant clopen subsets of the set  $U((x_{n+1}^{(1)}$, $x_{n+2}^{(1)}$, $\ldots$, $x_N^{(1)}))$. Introduce elements $t_j\in G$, $j\in \mathbb N$, as follows. For $x\in X_B$ set
 $$ t_j(x)=\left\{\begin{array}{ll}s(x),& x\in U_j\cup s(U_j),\\ x,&\text{otherwise}.\end{array}\right.$$ Then $t_j^{-1}ht_j(x)$ starts with the prefix $(y_1^{(2)},y_2^{(2)},\ldots,y_N^{(2)})$ for $x\in U_j\cap U((x_1^{(1)},\ldots,x_N^{(1)}))$ and with the prefix $(y_1^{(1)},y_2^{(1)},\ldots,y_N^{(1)})$ for $x\in U((x_1^{(1)},\ldots,x_N^{(1)}))\setminus U_j$.
 It follows that $t_j^{-1}ht_j$ are pairwise disjoint.\end{proof}
 %%%%%%%%%%%%%%%%%%%%%%%%%%%%%%%%%%%%%%%%%%%%%%%%%%%%%%%%%%%%%%%%%% 

 \noindent Lemma \ref{lemma: fpc G_n} implies that
 \begin{equation}\label{eq: fpc G_n}\fpc(g)\subset G_n\;\;\text{for any}\;\;g\in G_n.\end{equation}

% \begin{equation}\label{eq: commutant formula 1}
% L(G_{n,\infty})'\cap L(G)\subset L(G_n)\;\;\text{for any}\;\;n\in\mathbb N.
% \end{equation}
 However, there is also another useful way of defining ``complementary" subgroups in $G$. Let $\suppg(g)$ stand for the support of the action of an element $g\in G$ on $X$ (not to be confused with $\suppe$). Let us define for a clopen set $A$ by $G_A$ the subgroups of all elements in $G$ supported inside $A$: $G_A=\{g\in G:\suppg(g)\subset A\}$. Clearly, $G_A$ commutes with $G_{X\setminus A}$. In fact, $G_{X\setminus A}$ is the centralizer of $G_A$ for each $A$.
  \begin{lemma}\label{lemma: fpc G_A} Let $A\subset X$ be a clopen set. Then for any $h\notin G_A$ one has: $\sharp\{t^{-1}ht\mid t\in G_{X\setminus A}\}=\infty$.
 \end{lemma}
\begin{proof} Let $h\notin G_A$. Then there exists $x\in X$ such that $hx\neq x$ and $hx\in X\setminus A$. Let $Y\subset X\setminus A$ be a clopen set such that $hx\in Y,x\notin Y$. The subgroup $G_Y\subset G$ of elements supported on $Y$ is the full group of a Bratteli diagram $\widetilde B$ formed by edges and vertices of all infinite paths $(e_1,e_2,\ldots)$ belonging to $Y$. By our assumption the initial diagram $B$ is simple (see Remark \ref{remark: full group}). This implies  that the diagram $\widetilde B$ is simple as well, and so the action of $G_Y$  on $Y$ is minimal. In particular, the orbit of $hx$ under $G_Y\subset G_{X\setminus A}$ is infinite. Let $\{g_i\}\in G_Y$ be a sequence of elements such that the points $g_ihx$ are pairwise distinct. Observe that $g_i^{-1}$ acts on $x$ trivially, since $x\notin Y$, and so $g_ihg_i^{-1}x=g_ihx$. We obtain that the elements $g_ihg_i^{-1}$ are pairwise distinct. This finishes the proof.
\end{proof}

\noindent We obtain that
\begin{equation}\label{eq: fpc G_A}\fpc(g)\subset G_{\suppg(g)}\;\;
\text{for any}\;\;g\in G.\end{equation}
\begin{remark}[Alternative proof of  \eqref{eq: fpc G_A}] 
It is straightforward to verify that full groups of simple Bratteli diagrams satisfy the conditions of Lemma \ref{lemma: fpc preserves orbits}. Moreover, all $g\in G$ are of finite order. If $\suppg(h)\not\subset \suppg(g)$ then $h$ does not preserve orbits of $g$ and so $h\notin\fpc(g)$ by Lemma \ref{lemma: fpc preserves orbits}. 
\end{remark}

Recall that for any $n\in\mathbb N$ the group $G_n$ is equal to the direct product
$\prod_{v\in V_n}G_v,$ where $G_v$ is the group of all permutations of paths joining the root vertex $v_0$ with $v$. A transposition in $G_n$ is an element interchanging two different paths from $v_0$ to some $v\in V_n$ and acting trivially on all other paths. Fix a $G$-invariant von Neumann subalgebra $P\subset L(G)$ and let $E:L(G)\to P$ be the conditional expectation preserving the trace $\tau(m)=\langle m\delta_e,\delta_e\rangle$. \begin{lemma}\label{lemma: E(permutations)} For any $g\in G$ with $g^2=e$ one has either $E(g)=g$ or $E(g)=0$.
\end{lemma}
\begin{proof} Observe that for any $n\in\mathbb N$ any $g\in G_n$ with $g^2=e$ can be written as a product of disjoint transpositions from $G_n$. Thus, we need to show that for any $n,k\in\mathbb N$ and any collection of $k$ disjoint transpositions $s_1,\ldots,s_k\in G_n$ one has: either $E(s_1\cdots s_k)=0$ or $E(s_1\cdots s_k)=s_1\cdots s_k$.
For every $n$ we will prove the statement by induction on $k$.
\vskip 0.2cm\noindent
 {\bf Base of induction: $k=1$} By \eqref{eq: fpc G_n} and \eqref{eq: fpc G_A} we have for any $s\in G_n$: $\fpc(s)\subset G_n\cap G_{\suppg(s)}$. %Let $v\in V_n$ and $s\in G_v$ be any transposition of two paths joining $v_0$ and $v$.
 Let $s\in G_n$ be a transposition. It is not hard to see that $G_n\cap G_{\suppg(s)}=\{e,s\}$. It follows that $\fpc(s)=\{e,s\}$. Thus, $E(s)=\lambda+\mu s$. Since $E$ preserves the trace $\tau$ we have $\lambda=\tau(E(s))=\tau(s)=0$. Since $E(E(s))=E(s)$ we obtain that $\mu^2 s=\mu s$. From the above we deduce that either $E(s)=s$ or  $E(s)=0$.
\vskip 0.2cm\noindent
{\bf Step of induction.} Assume that the statement is true for any number of disjoint transpositions less than $k$, where $k>1$. Let $s_1,\ldots, s_k\in G_n$ be $k$ pairwise disjoint transpositions.
\vskip 0.2cm\noindent
Case (i): there exists a non-empty subset $S\subsetneq I_k=\{1,\ldots,k\}$  such that $E(\prod_{i\in S}s_i)=\prod_{i\in S}s_i$. Then one has:
$$E(\prod_{i\in I_k}s_i)=E(\prod_{i\in S}s_i\cdot\prod_{i\in I_k\setminus S}s_i)=\prod_{i\in S}s_i\cdot E(\prod_{i\in I_k\setminus S}s_i),$$ since $\prod_{i\in S}s_i\in P$. By induction, $E(\prod_{i\in I_k\setminus S}s_i)$ is equal to either $0$ or $\prod_{i\in I_k\setminus S}s_i$ and the claim of induction follows.
\vskip 0.2cm\noindent
Case (ii): for every non-empty subset $S\subsetneq I_k=\{1,\ldots,k\}$  one has $E(\prod_{i\in S}s_i)=0$. Using \eqref{eq: fpc G_n} and \eqref{eq: fpc G_A} and Lemma \ref{lemma: fpc preserves orbits} we obtain that $\fpc(\prod_{i\in I_k}s_i)\subset \{\prod_{i\in S}s_i:S\subset I_k\}$. We obtain that 
$$E(\prod_{i\in I_k}s_i)=\sum\limits_{S\subset I_k}\lambda_S \prod_{i\in S}s_i,$$ where $\lambda_S$ are some constants. Since $E$ preserves $\tau$ we have $\lambda_\varnothing=0$.  Applying $E$ to the above formula we obtain that $\lambda_S=0$ for all $S\subsetneq I_k$ and $\lambda_{I_k}^2=\lambda_{I_k}$. It follows that either $E(\prod_{i\in I_k}s_i)=0$ or $E(\prod_{i\in I_k}s_i)=\prod_{i\in I_k}s_i$.
\end{proof}
\begin{lemma}\label{lemma: P contains G'}
One has either $P=\mathbb C$ or $P\supset L(G')$.
\end{lemma}
 \begin{proof} Let $v\in V_n$ for $n$ sufficiently large such that the number of paths between $v_0$ and $v$ is at least 5. In particular, $G_v$ contains two disjoint transpositions $s_1,s_2$. By Lemma \ref{lemma: E(permutations)} there are two possibilities. 
\vskip 0.2cm\noindent
Case (i): $E(g)=g$ for some $g\in G,g\neq e$, such that $g^2=e$. Since the set $H=\{h\in G:E(h)=h\}$ is a non-trivial normal subgroup, the center of $G$ is trivial, and $G'$ is simple we obtain that $H>G'$ and so $P\supset L(G')$.
\vskip 0.2cm\noindent
 Case (ii): $E(g)=0$ for all $g\in G,g\neq e$, such that $g^2=e$. Consider the character $\phi(g)=\tau(E(g)g^{-1})$ on $G$. Fix any $g\in G,g\neq e$. Let $n\in\mathbb N$ be such that $g\in G_n$. There exists $v\in V_n$ and a finite path $p\in E(v_0,v)$ such that $gp\neq p$. Find a sequence of elements $s_k$, $k\in\mathbb N$, with $$s_k^2=e,\;s_k\neq e,\; \suppg(s_k)\subset X_v^{(n)}(p),\;\;\text{and}\;\;\suppg(s_k)\cap \suppg(s_m)=\varnothing$$ for every $m\neq k$, $k,m\in\mathbb N$. Consider the elements $h_{k,m}=(s_kgs_k^{-1})\cdot (s_mgs_m^{-1})^{-1}$, $k\neq m$.  One can write:
 $$h_{k,m}=s_k\cdot gs_kg^{-1}\cdot gs_mg^{-1}\cdot s_m.$$ By construction, the elements $s_k,s_m, gs_kg^{-1},gs_mg^{-1}$ are of order two and have pairwise disjoint supports (with the latter two supported inside $gX_v^{(n)}(p)=X_v^{(n)}(gp)$). It follows that $h_{k,m}^2=e$ and $h_{k,m}\neq e$, for every $k\neq m$. 
 From the conditions of case $(ii)$ we obtain that  $\phi(h_{k,m})=0$ whenever $k\neq m$. Using Lemma \ref{lemma: vanishing character from many conjugates} we conclude that $\phi(g)=0$.
 % and write $$\phi(g)=\int_{Ch(G)}\chi_{\alpha}(g)d\mu(\chi_{\alpha}),$$
%  where $\chi_\alpha$ ranges over indecomposable characters on $G$. Each of the indecomposable characters on $G$ is either regular or has the form
%  \begin{equation*}\chi(g)=\rho([g])\prod_{i=1}^{k} \mu_i(\fix(g))^{\alpha_i}
%  \end{equation*} for some $k\geqslant 0$, $\alpha_i\in\mathbb N$, ergodic $G$-invariant measures $\mu_i$ on $X$, and a homomorphism $\rho:G/G'\to\mathbb C$ (see \cite{dm_ggd2014} for the case when $G$ is simple and admits only finitely many invariant ergodic measures on $X$, \cite{dm_manuscript}  for the general case). It follows that each non-regular indecomposable character on $G$ has positive value on every element $g\in G'$ with $\suppg(g)\neq X$ (in particular, on $s_1s_2$). Since $\phi(s_1s_2)=0$, we deduce that $\mu$ is the Dirac measure of the regular character and so
 Thus, $\phi$ is the regular character. We deduce that $E(g)=0$ for all $g\neq e$. It follows that $P=\mathbb C$.
 \end{proof}
 Let $g\in G$ and let $n\in\mathbb N$ be such that $g\in G_n$. Since $G_n=\prod_{v\in V_n}G_v$ is a direct product of a finite number of finite symmetric groups, there exists a subset $S\subset V_n$ and a collection of transpositions $s_v\in G_v,v\in S$, such that $g=\prod_{v\in S}s_v\cdot h,$ where $h\in G_n'<G'$. By Lemma \ref{lemma: P contains G'} either $P=\mathbb C$ or $P\supset L(G')$. In the latter case we have: $E(g)=E(\prod_{v\in S}s_v)h$. From Lemma \ref{lemma: E(permutations)} we derive that either $E(g)=0$ or $E(g)=\prod_{v\in S}s_v\cdot h=g$. Since this is true for all $g\in G$, by Proposition \ref{prop: criterion on P=L(H)} we obtain that $P=L(H)$ for some normal subgroup $H<G$.

\section{Groups with the ISR property from modified character approach}

In this section we construct various groups with the ISR property inspired by the character approach. Note that in all these examples, the non-factorizable character property does not hold.

\subsection{Checking ISR for groups without non-factorizable regular characters}\label{subsection: ISR groups without the factorizable property}

In this subsection, we prove the ISR property for  certain classes of groups without non-factorizable regular characters by exploring the structure of normal subgroups in them.

First, recall that in  \cite[Question 4.3]{aj}, it was  asked whether the ISR property for infinite groups is preserved under taking product. We expect that, given any finitely many non-amenable groups with trivial amenable radical and non-factorizable regular characters, their product group has the ISR property. Let us present some partial results in this direction.  

Recall that a group $G$ is called character rigid it the only indecomposable characters on $G$ are the regular one and the trivial one: $\IChar(G)=\{\delta_e,\textbf{1}_G\}$.
\begin{proposition}\label{prop: product of two groups with character rigidity has the ISR property}
Let $G$ and $H$ be infinite non-amenable character rigid groups. Then $G\times H$ has the ISR property.
\end{proposition}
\begin{proof}
Recall that every extremal character on $G\times H$ splits as a product of extremal characters from $G$ and $H$ \cite[Corollary 2.13]{bfr}. Let $P\subseteq L(G\times H)$ be a $G\times H$-invariant von Neumann subalgebra. Note that since $G\times H$ has trivial amenable radical, we can deduce that $P$ is in fact a subfactor by \cite[Theorem A]{aho}. Repeating steps in the proof of Theorem \ref{theorem: NFRC consequences} (see also the proof of \cite[Theorem 3.1]{cds}), we may find some normal subgroup $N\lhd G\times H$ such that $L(N)=P\bar{\otimes}Q$, where $Q=P'\cap L(N)$.

First, assume that $N\cap G=\{e\}$ or $N\cap H=\{e\}$.\footnote{In fact, we can show that $N=\{e\}, G, H$ or $G\times H$. To prove that $P$ is a subgroup factor, we may assume that $N=G\times H$ since otherwise, it is clear that $P$ is of the desired form. However, we prefer to use an alternative approach, since it generalizes to a wider setting.}
Suppose that $N\cap G=\{e\}$ by symmetry. Then, given $(s, t)\in N$, the element $s\in G$ is uniquely determined by the element $t\in H$ and hence we may denote it by $\phi(t)$. It follows that $N=\{(\phi(h), h): h\in \pi_2(N)\}$, where $\pi_1: G\times H\rightarrow G$ and $\pi_2: G\times H\rightarrow H$ denote the coordinate projections. 

It is not hard to check that $\phi: \pi_2(N)\twoheadrightarrow \pi_1(N)$ is a group homomorphism. The cases $N=\{e\}$ and $N=H$ are trivial, so we will assume that $N\neq\{e\}$ and $N\neq H$. Since $\pi_1(N)\lhd G$, $\pi_2(N)\lhd H$ and $G, H$ are simple groups, we have that $\pi_1(N)=G$ and $\pi_2(N)=H$.
% Otherwise, $N=\{e\}$ or $H$. In either case, we know that $P$ is a subgroup factor and hence the proof is finished. 
It follows that $N=\{(\phi(h), h): h\in H\}$. Since $N\lhd G\times H$, we deduce that for all $h\in H$, $\phi(h)\in Z(G)$ (the center of $G$). However, $Z(G)$ is trivial, since $G$ is non-abelian and simple. This implies that $N=H$ and finishes the proof in this case.

Now, assume that $N\cap G=\{e\}$ and $N\cap H=\{e\}$. Denote by $\phi(g)=\tau(E(g)g^{-1})$ and $\psi(g)=\tau(E'(g)g^{-1})$ for all $g\in G\times H$, where $E: L(G\times H)\rightarrow P$ and $E': L(G\times H)\rightarrow Q$ are the trace preserving conditional expectations. Note that both $\phi$ and $\psi$ are characters on $G\times H$ and we have shown before (while proving Theorem \ref{theorem: NFRC consequences}) that $\phi(s)\psi(s)=0$ for all $e\neq s\in N$.

By \cite[Theorem 2.12]{bfr} we may write $\phi=a\delta_{e}\otimes \delta_{e}+b1_G\otimes \delta_{e}+c\delta_{e}\otimes 1_H+d1_G\otimes 1_H$ and $\psi=a'\delta_{e}\otimes \delta_{e}+b'1_G\otimes \delta_{e}+c'\delta_{e}\otimes 1_H+d'1_G\otimes 1_H$. Here, $a,b,c,d,a',b',c',d'$ are all non-negative real numbers with $a+b+c+d=1=a'+b'+c'+d'$.

By our assumption, we may find $(s, e), (e, t)\in N$ with $s\neq e$ and $t\neq e$. Now, we may plug in $g=(s, e), (e, t)$ and $(s, t)$ respectively into the formula for $\phi(g)$ and $\psi(g)$, then use $0=\phi(g)\psi(g)$ for all $e\neq g\in N$ to deduce the following identities.
\begin{align*}
    0&=(b+d)(b'+d'),\\
    0&=(c+d)(c'+d'),\\
    0&=dd'.
\end{align*}

If $b+d=0$, equivalently, $b=d=0$, then $\phi=a\delta_e\otimes \delta_e+c\delta_e\otimes 1_H$. This implies that $\phi(gh)=0$ and hence $E(gh)=0$ for all $e\neq g\in G$ and all $h\in H$. Therefore, $P\subseteq L(H)$. Hence $P$ is a subgroup factor since $H$ has the ISR property.

If $b'+d'=0$, equivalently, $b'=d'=0$, then we deduce from the second equation above that either $c=d=0$ or $c'=d'=0$. When $c=d=0$, we can argue similarly as above to deduce that $P\subseteq L(G)$ and hence $P$ is a subgroup factor since $G$ has the ISR property. When $b'=d'=c'=0$, then $\psi=\delta_e\otimes \delta_e$ and hence $Q=\mathbb{C}$, thus $P=L(N)$.
\end{proof}
\begin{remark}
    When trying to generalize this proposition to the general setting, \ie two non-amenable groups with trivial amenable radical and non-factorizable regular character, and try to show the product has the ISR property, the main difficulty is that for a normal subgroup $N\lhd G\times H$, $(N\cap G)\times (N\cap H)$ may be a proper subgroup in $N$. It is well-known that the general structure for a subgroup in $G\times H$ is described by Goursat's lemma.
\end{remark}
With the base case considered above, we can do induction to get the following general result.
\begin{proposition}\label{prop: product of n groups with character rigidity has the ISR property}
Let $n\geq 2$ and let $G_i$ be an infinite non-amenable character rigid group for every $1\leq i\leq n$. Then $G:=\oplus_{1\leq i\leq n}G_i$ has the ISR property.
\end{proposition}
\begin{proof}
We prove it by doing induction on $n$, the number of summands in the definition of $G$.

For $n=2$, this was proved in previous Proposition \ref{prop: product of two groups with character rigidity has the ISR property}.
Assume that the conclusion holds when the number of summands in $G$ is less than $n$, and let us prove it for $G$ with $n$-many summands. Let $P\subseteq L(G)$ be a $G$-invariant von Neumann subalgebra. We know that $P$ is in fact a subfactor by \cite[Theorem A]{aho} since $G$ has trivial amenable radical. We aim to show $P$ is a subgroup subfactor. 
As mentioned before (in proving Theorem \ref{theorem: NFRC consequences}, see also the proof of \cite[Theorem 3.1]{cds}), there exists some normal subgroup $N\lhd G$ with $L(N)=P\bar{\otimes} Q$, where $Q=P'\cap L(N)$.
%Claim: We may assume that $N\cap G_i\neq \{e\}$ for all $1\leq i\leq n$.

First, suppose that for some $i$ we have $N\cap G_i=\{e\}$. Then we have $N=\{(\phi(h), h): h\in \pi'_i(N)\}$, where:
\begin{itemize}
 \item $\phi: \pi'_i(N)\twoheadrightarrow\pi_i(N)$ is a group homomorphism.
    \item $\pi_i: G\twoheadrightarrow G_i$ denotes the projection onto the $i$-th coordinate.
    \item $\pi'_i: G\twoheadrightarrow \cdots\oplus G_{i-1}\oplus G_{i+1}\oplus\cdots:=G\ominus G_i$ is the projection onto all coordinates except the $i$-th.
\end{itemize}
%Note that here we have abused notations by swapping coordinates/identifying $G$ with $(G\ominus G_i)\oplus G_i$.
\noindent Since $N\lhd G$, we obtain that $\phi(h)$ belongs to the center $Z(G_i)$ of $G_i$, which is trivial. Thus, $\phi$ is trivial. This implies that $N=\pi_i'(N)\subseteq G\ominus G_i$. Using the induction hypothesis, we conclude that $P$ is of the desired form, since now $P\subseteq L(G\ominus G_i)$.

Thus, we may assume that $N\cap G_i\neq \{e\}$ for all $1\leq i\leq n$. Choose $e\neq s_i\in N\cap G_i$ for each $1\leq i\leq n$. Let $E: L(G)\rightarrow P$ and $E': L(G)\rightarrow Q$ be the trace preserving conditional expectations. Define $\phi(g)=\tau(E(g)g^{-1})$ and $\psi(g)=\tau(E'(g)g^{-1})$ for all $g\in G$. Note that we have shown before (in proving Theorem \ref{theorem: NFRC consequences}) that $\phi(g)\psi(g)=0$ for all $e\neq g\in N$ and $\phi, \psi$ are both characters on $G$.

By \cite[Corollary 2.13]{bfr}, we know that every indecomposable character on $G$ is a product of indecomposable characters on each $G_i$. Therefore, we may write 
\begin{align}\label{eq: phi=sum psi_epsilon}\begin{split}
    \phi&=\sum_{(\epsilon_i)_i\in \{0, 1\}^n}c_{\epsilon_1,\ldots, \epsilon_n}1_{\epsilon_1,\ldots, \epsilon_n},\\
    \psi&=\sum_{(\epsilon_i)_i\in \{0, 1\}^n}d_{\epsilon_1,\ldots, \epsilon_n}1_{\epsilon_1,\ldots, \epsilon_n},
\end{split}    
\end{align}
where all the coefficients $c_{\epsilon_1,\ldots, \epsilon_n}, d_{\epsilon_1,\ldots, \epsilon_n}$ are non-negative real numbers with $\sum_{(\epsilon_i)_i\in \{0, 1\}^n}c_{\epsilon_1,\ldots, \epsilon_n}=1=\sum_{(\epsilon_i)_i\in \{0, 1\}^n}d_{\epsilon_1,\ldots, \epsilon_n}$. In formula \eqref{eq: phi=sum psi_epsilon},
$1_{\epsilon_1,\ldots, \epsilon_n}$ stands for the indecomposable character $1_{\epsilon_1}\otimes\cdots\otimes 1_{\epsilon_n}$ and for any $1\leq i\leq n$ we set
$1_{\epsilon_i}=\begin{cases}\delta_e,& \epsilon_i=0,\\
1_{G_i}, &\epsilon_i=1\end{cases}$. In other words, for any $(g_1,\cdots, g_n)\in G$, we have
\begin{align*}
    1_{\epsilon_1,\ldots, \epsilon_n}(g_1,\ldots, g_n):=\begin{cases}
    1, & g_i=e~\mbox{if $\epsilon_i=0$, $\forall i$}\\
    0, & otherwise.
\end{cases}
\end{align*}

Next, suppose that for each $1\leq j\leq n$, there exists some $p_j\in \{0, 1\}^n$ such that $c_{p_j}>0$ and the $j$-th coordinate of $p_j$ is equal to 1. For each $1\leq j\leq n$, we have that
\begin{align*}
    \phi(s_j)=\sum_{(\epsilon_i)_i\in \{0, 1\}^n}c_{\epsilon_1,\ldots, \epsilon_n}1_{\epsilon_1,\ldots, \epsilon_n}(s_j)=\sum_{(\epsilon_i)_i\in \{0, 1\}^n, \epsilon_j=1}c_{\epsilon_1,\ldots, \epsilon_n}\geq c_{p_j}>0, 
\end{align*} where $s_j\in N\cap G_j\setminus \{e\}$, as was defined earlier. 
From $\phi(s_j)\psi(s_j)=0$, we deduce that
\begin{align*}
    0=\psi(s_j)=\sum_{(\epsilon_i)_i\in \{0, 1\}^n, \epsilon_j=1}d_{\epsilon_1,\ldots, \epsilon_n}.
\end{align*}
Equivalently, $d_{\epsilon_1,\ldots, \epsilon_n}=0$ for all $(\epsilon_i)_i\in \{0, 1\}^n$ with $\epsilon_j=1$. Since $j$ is arbitrary, this implies that
\begin{align*}
    \psi=1_{0,\cdots, 0}=\delta_e.
\end{align*}
Therefore, $E'(g)=0$ for all $e\neq g\in G$, hence $Q=\mathbb{C}$, and thus $P=L(N)$, which finished the proof for this case. 

Thus, we may assume that there exists some $j$ such that $c_{\epsilon_1,\ldots, \epsilon_n}=0$ for all $(\epsilon_i)_i\in \{0, 1\}^n$ with $\epsilon_j=1$. Then one has:
\begin{align*}
    \phi=\sum_{(\epsilon_i)_i\in \{0, 1\}^n, \epsilon_j=0}c_{\epsilon_1,\ldots, \epsilon_n}1_{\epsilon_1,\ldots, \epsilon_n}.
\end{align*}
This implies that for any $g=(t_1,\ldots, t_n)\in G$ with $t_j\neq e$ we have $\phi(g)=0$ (equivalently, $E(g)=0$). Hence,  $P\subseteq L(G\ominus G_j)$ and the proof is finished by the induction hypothesis.
\end{proof}

Next,  we show that some generalized wreath product groups have the ISR property. Recall that an action of a group is called faithful if the only group element acting trivially is the identity element.
 
\begin{proposition} \label{prop: ISR for certain generalized wreath product groups}
Let $G=H\wr_I K$, where $H$ is an infinite ICC simple group with non-factorizable regular character, $K$ is an infinite countable group and $K\curvearrowright I$ is a faithful and transitive action on a non-empty set $I$. 
Then every $G$-invariant von Neumann subfactor $P$ in $L(G)$ is a subgroup subfactor. If we further assume $H$ is non-amenable, then $G$ has the ISR property.
\end{proposition} 
 \begin{proof}
 We start by showing that for every normal subgroup $N$ in $G$, if $N\neq \{e\}$, then $N=H\wr_I K_0=\oplus_IH\rtimes K_0$ for some normal subgroup $K_0$ in $K$.
 
 To see this, first note that $N\cap \oplus_IH\neq \{e\}$. Indeed, otherwise, we get that $N$ commutes with $\oplus_IH$, \ie $N\subseteq C_G(\oplus_IH)$, the centralizer of $\oplus_IH$ in $G$. It is clear that $C_G(\oplus_IH)\subseteq \oplus_IH$ since $K\curvearrowright I$ is faithful. Thus $N\cap \oplus_IH=N=\{e\}$, which contradicts to our assumptions. Then, pick any non-trivial $n:=\oplus_{s\in F}h_s\in N\cap \oplus_IH$, where $\emptyset \neq F\subset I$ and $h_s\in H\setminus \{e\}$ for all $s\in F$. Fix any $s_0\in F$. Since $H$ is a simple group, we get that $H$ has trivial center and hence there exists some $h\in H$ such that $hh_{s_0}h^{-1}\neq h_{s_0}$. Denote by $H_{s_0}<\oplus_IH$ the isomorphic copy of the group $H$ at the coordinate $s_0$.  Since $N$ is normal in $G$, we get that $n':=hnh^{-1}\in N$, where $h\in H_{s_0}$ is understood as the element $h$ at coordinate $s_0$ in $\oplus_IH$. Thus $e\neq n'n^{-1}=(hh_{s_0}h^{-1})h_{s_0}^{-1}\in N\cap H_{s_0}$. Hence $N\cap H_{s_0}=H_{s_0}$ since $H$ is simple. Clearly, this implies that $\oplus_IH\subseteq N$ since $K$ acts transitively on $I$.  Then from $\oplus_IH\subseteq N\subseteq G$, we deduce that $N=\oplus_IH\rtimes K_0$ for some normal subgroup $K_0$ in $K$.
 
(1) Let $P\subseteq L(G)$ be a $G$-invariant von Neumann subfactor. By arguing similarly as in proving Theorem \ref{theorem: NFRC consequences},  we get some normal subgroup $N$ in $G$ with $L(N)=P\bar{\otimes} (P'\cap L(N))$. Set $\phi(s):=\tau(s^{-1}E(s))$ and $\psi(s)=\tau(s^{-1}E'(s))$ where $s\in G$, $E: L(G)\rightarrow P$ and $E': L(G)\rightarrow P'\cap L(N)$ denote the trace preserving conditional expectations. Then $\phi(s)\psi(s)=0$ for all $e\neq s\in N$.

If $N=\{e\}$, then $P=\mathbb{C}$ and we are done. So we may assume that $N\neq \{e\}$. By what we have proved, $N=\oplus_IH\rtimes K_0$ for some normal subgroup $K_0$ in $K$. By considering $\phi|_{H_s},\psi|_{H_s}$, where $H_s$ denotes the copy of $H$ in $\oplus_IH$ at some fixed coordinate $s\in I$, we deduce that $\phi|_{H_s}\equiv \delta_e$ or $\psi|_{H_s}\equiv \delta_e$ from the assumption on $H$.
 
Case 1. $\phi|_{H_s}\equiv \delta_e$. 
 
Given any $g=\oplus g_i\in \oplus_IH$ with $g_s\neq e$ and any $k_0\in K_0$, we claim that $\phi(gk_0)=0$.

Indeed, since $H$ is infinite ICC, there are infinitely many distinct $h_n\in H$ with $h_ng_sh_n^{-1}\neq h_mg_sh_m^{-1}$ for any $n\neq m$. Identify $h_n$ with the corresponding element of $H_s$ for each $n\in\mathbb N$.  Then $g_n:=h_n(gk_0)h_n^{-1}$ are pairwise distinct elements in $\oplus_IH\rtimes K_0$ by the following calculation (where we write $\sigma$ for the action $K\curvearrowright I$):
\begin{align*}
g_n^{-1}g_m&=(h_ngk_0h_n^{-1})^{-1}(h_mgk_0h_m^{-1})\\
&=h_n\sigma_{k_0^{-1}}(g^{-1}h_n^{-1}h_mg)h_m^{-1}\\
&=h_n\sigma_{k_0^{-1}}(g_s^{-1}h_n^{-1}h_mg_s)h_m^{-1}\\
&=\begin{cases}
(h_ng_s^{-1}h_n^{-1})(h_mg_sh_m^{-1})\in H_s\setminus \{e\}~&\text{if $k_0^{-1}s=s$}\\
h_nh_m^{-1}\oplus g_s^{-1}h_n^{-1}h_mg_s\in H_s\oplus H_{k_0^{-1}s}\setminus\{e\}~&\text{if $k_0^{-1}s\neq s$}.
\end{cases}
\end{align*}

Subcase 1. $k_0^{-1}s=s$.

Then $\phi(g_n^{-1}g_m)=0$ for all $n\neq m$ since $\phi|_{H_s}=\delta_e$. By the Lemma \ref{lemma: vanishing character from many conjugates}, we deduce that $\phi(gk_0)=\phi(g_n)=0$. 
Note that this implies $\left\|E(gk_0)\right\|_2^2=\tau((gk_0)^{-1}E(gk_0))=\phi(gk_0)=0$, hence $E(gk_0)=0$. 

Subcase 2. $k_0^{-1}s\neq s$.

It suffices to show that $\phi|_{H_s\oplus H_{k_0^{-1}s}}=\delta_{e}$. To see this,
for any $e\neq h\in H_{k_0^{-1}s}$, we have $k_0hk_0^{-1}\in H_s$ and hence $\phi(h)=\phi(k_0hk_0^{-1})=0$
since $\phi|_{H_s}=\delta_s$. Thus $\phi|_{H_{k_0^{-1}s}}=\delta_e$. Next, consider any $\tilde{g}:=h\oplus h'\in H_s\oplus H_{k_0^{-1}s}$ with $h\neq e\neq h'$. Since $H$ is ICC, we may find an infinite sequence $s_n\in H$ such that $s_nhs_n^{-1}\neq s_mhs_m^{-1}$ for all $n\neq m$. Identify $s_n$ with the corresponding element in $H_s$ for each $n$ and set $\tilde{g}_n=s_n\tilde{g}s_n^{-1}$. Then $\phi(\tilde{g}_n^{-1}\tilde{g}_m)=\phi((s_nhs_n^{-1}\oplus h')^{-1}(s_mhs_m^{-1}\oplus h'))
=\phi((s_nhs_n^{-1})^{-1}(s_mhs_m^{-1}))=0$ for all $n\neq m$. Hence $\phi(\tilde{g})=\phi(\tilde{g}_n)=0$. 

In other words, $\text{supp}(P)\subseteq \{gk: g\in \oplus_IH, g_s=e, k\in K_0\}$, where $\text{supp}(P)=\{\text{supp}(x): x\in P\}$ and, as before, $\text{supp}(x)=\{g\in G: \tau(xg^{-1})\neq 0\}$, \ie the collection of all coordinates with non-zero corresponding Fourier coefficients in the Fourier expansion of $x$. Since $P$ is $K$-invariant, it is not hard to see that in fact $\text{supp}(P)\subseteq K_0$ as $K\curvearrowright I$ is transitive. Since $P$ is also $\oplus_IH$-invariant, we can further deduce that $\text{supp}(P)\subseteq \{e\}$ from the faithfulness of $K\curvearrowright I$; equivalently, $P=\mathbb{C}$.

Case 2. $\psi|_H\equiv \delta_e$.  

We can argue similarly as above to deduce that $P'\cap L(N)=\{e\}$. Hence $L(N)=P$.

The proof of the first part is done.

(2) By \cite{aho}, it suffices to check that $G$ has trivial amenable radical, which follows easily from the above characterization of non-trivial normal subgroups in $G$.
 \end{proof}
 
 \begin{remark} 
 Note that for the above $G=H\wr_IK$, it does not have the non-factorizable regular character property since it contains the normal subgroup $\oplus_IH$ which lacks this property by Proposition \ref{prop: NFRC basic properties}.
 \end{remark}

\subsection{The ISR property does not imply triviality of the amenable radical}\label{subsection: non-trivial amenable radical}

In this subsection, we observe that there are groups with the ISR property having non-trivial amenable radical.
\begin{proposition}
Let $G=A_{\infty}$, the finitary alternating group on $\mathbb{N}$, \ie the group of all even permutations of $\mathbb{N}$ with finite supports. Then $G$ satisfies the ISR property.
\end{proposition}
\begin{proof}
The proof is essentially the same as the one used in showing $G=S_{\infty}$ has the ISR property in \cite{jz}. Possible modifications we need to make include that whenever we use some $s\in S_{\infty}$ to do conjugation on an element in the group ring of $S_{\infty}$, \eg $E((1~2~3))$ therein, we may actually assume that $s\in A_{\infty}$ (after replacing $s$ by $(i~j)s$ for some large enough $i,j$ if necessary) without changing the final result.
Moreover, note that indecomposable characters on $A_{\infty}$ are the restriction of those on $S_{\infty}$ \cite{thoma_mz}, see also \cite[Section 6]{thomas_etds}.
\end{proof}

\begin{theorem}\label{prop: example of nonamenable groups with ISR but nontrivial amenable radical}
Let $B$ be any non-amenable group with only two conjugacy classes (see \cite{osin_annals} for concrete examples). Then
$G=A_{\infty}\times B$ has the ISR property.
\end{theorem}
\begin{remark}
    It would be clear from the proof that $A_{\infty}$ can be replaced by any infinite ICC simple amenable group with the ISR property. 
\end{remark}
\begin{proof}
First, we note that the following properties hold true.

\begin{itemize}
    \item $B$ is a simple ICC group with trivial amenable radical.
    \item $B$ has the ISR property.
    \item For any normal subgroup $H\lhd G$, we have either $H=A_{\infty}$ or $H=B$ or $H=G$ or $H=\{e\}$.
\end{itemize}

The first item is clear and the second item could be proven using Proposition \ref{prop: examples of NSFC-simple groups} since infinite groups with only two conjugacy classes are character rigid. 
To see that the last item holds true, note that $H\cap A_{\infty}\lhd A_{\infty}$, which is simple, thus $H\cap A_{\infty}=\{e\}$ or $A_{\infty}$. If $H\cap A_{\infty}=A_{\infty}$, then $A_{\infty}\subseteq H$ and hence $H=A_{\infty}\times K$ for some $K\lhd B$. Since $B$ is simple, we get that $K=\{e\}$ or $K=B$. In other words, in this case, either $H=A_{\infty}$ or $H=A_{\infty}\times B$. 

Assume now that $H\cap A_{\infty}=\{e\}$. By symmetry, \ie by considering whether $H\cap B=\{e\}$ or $B$, we may assume that $H\cap B=\{e\}$. Then, it is not hard to see that $H=\{(s, \phi(s)): s\in p_1(H)\}$ for some group isomorphism $\phi: p_1(H)\cong p_2(H)$, where $p_1: G\twoheadrightarrow A_{\infty}$ is the coordinate projection onto the first coordinate and $p_2$ is coordinate projection onto the second coordinate. Note that $p_1(H)\lhd A_{\infty}$ and $p_2(H)\lhd B$, therefore, we get both $p_1(H)$  and $p_2(H)$ are trivial and thus $H=\{e\}$.

Next, let us prove $G$ has the ISR property.

Let $P\subseteq L(G)$ be a $G$-invariant von Neumann subalgebra. Then, the center $\mathcal{Z}(P)$ is also $G$-invariant and $\mathcal{Z}(P)\subseteq L(A_{\infty})$ by \cite[Theorem A]{aho}, since $A_{\infty}$ is the amenable radical for $G$. Since $A_{\infty}$ is a simple non-abelian group with the ISR property, we deduce that $\mathcal{Z}(P)=\mathbb{C}$. In other words, $P$ is a subfactor.
Recall that in \cite[Theorem 3.1]{cds}, it was proved that there exists a normal subgroup $H\lhd G$ such that $L(H)=P\bar{\otimes}(P'\cap L(H))$. We split the proof by considering 4 cases.

Case 1: $H=\{e\}$. Clearly, $P=\mathbb{C}$.

Case 2: $H=A_{\infty}$. Then, using that $P\subseteq L(H)=L(A_{\infty})$ is $A_{\infty}$-invariant and $A_{\infty}$ has the ISR property, we deduce that $P=\mathbb{C}$ or $P=L(A_{\infty})$.

Case 3: $H=B$. Similarly to Case 2, we deduce that $P=\mathbb{C}$ or $P=L(B)$.

Case 4: $H=G$. Let $E: L(G)\rightarrow P$ be the trace preserving conditional expectation. Set $\phi(g)=\tau(g^{-1}E(g))$ for all $g\in G$, which is a character on $G$.

Note that $P\cap L(A_{\infty})\subseteq L(A_{\infty})$ is $A_{\infty}$-invaraiant and similarly, $P\cap L(B)\subseteq L(B)$ is $B$-invariant. Hence from the ISR property for $A_{\infty}$ and $B$, we deduce that the following hold.
\begin{itemize}
    \item $P\cap L(A_{\infty})=\mathbb{C}$ or $P\cap L(A_{\infty})=L(A_{\infty})$.
    \item $P\cap L(B)=\mathbb{C}$ or $P\cap L(B)=L(B)$.
\end{itemize}
Note that if $P\cap L(A_{\infty})=L(A_{\infty})$, then $L(A_{\infty})\subseteq P\subseteq L(A_{\infty})\bar{\otimes} L(B)$. Then from the Ge-Kadison splitting theorem \cite{gk}, we deduce that $P=L(A_{\infty})\bar{\otimes}Q$ for some von Neumann subfactor $Q\subseteq L(B)$. Clearly, $Q$ is $B$-invariant, thus $Q=\mathbb{C}$ or $Q=L(B)$; equivalently, $P=L(A_{\infty})$ or $P=L(G)$. The case $P\cap L(B)=L(B)$ can be handled similarly. Therefore, without loss of generality, we may assume that $P\cap L(A_{\infty})=\mathbb{C}$ and $P\cap L(B)=\mathbb{C}$. It suffices to show  $P=\mathbb{C}$; equivalently, we need to show $\phi(g)=0$ for all $e\neq g\in G$.

Notice that for any $s\in A_{\infty}$, $E(s)\in L(B)'\cap L(G)=L(A_{\infty})$ and thus $E(s)\in L(A_{\infty})\cap P=\mathbb{C}$. Hence, $E(s)=\tau(E(s))=\tau(s)=\delta_{s,e}$; equivalently, $\phi(s)=0$ for all $e\neq s\in A_{\infty}$. Similarly, we can also deduce that $E(t)=0$ for all $e\neq t\in B$ and thus $\phi(t)=0$.

To sum up, we have shown that $\phi|_{A_{\infty}}=\delta_e$ and $\phi|_{B}=\delta_e$. 
Take any $e\neq g=(s, t)\in G$. Without loss of generality, we may assume that both $s\neq e$ and $t\neq e$. Then, notice that the assumption on $B$ implies $g$ is conjugate to $(s, t')$ for any $e\neq t'\in B$. Moreover, $\phi((s, t'')(s^{-1}, t'^{-1}))=\phi((e, t''t'^{-1}))=0$ for any $t''\neq t'$. Thus, $\phi(g)=0$ by Lemma \ref{lemma: vanishing character from many conjugates}. This finishes the proof.
\end{proof}

The above proposition shows that having trivial amenable radical is not a necessary condition for an ICC group to have the ISR property.

\begin{question}
Let $H\subseteq G$ be ICC groups with $[G:H]=2$. Is it true that $H$ has the ISR property iff $G$ has the ISR property?
\end{question}

 \section{More questions}\label{sec: more question}
 
In this section, we collect more questions motivated by our character approach to establish the ISR property.

It is generally believed that non-amenable groups with trivial amenable radical may have the ISR property, see \eg \cites{aho,cds}. While by Theorem \ref{prop: example of nonamenable groups with ISR but nontrivial amenable radical}, having trivial amenable radical is not a necessary condition for a non-amenable group to have the ISR property.
Despite various non-amenable groups are known to have the ISR property, the above conjecture is still open. Even worse, we do not even have a reasonable conjecture on which ICC amenable groups may have the ISR property.

Notice that we have proved the ISR property for full groups of simple Bratteli diagrams. An interesting problem would be to check whether it holds for topological full groups as well. Or even for classes of groups having a rich structure of normal subgroups, such as \emph{branch} and \emph{weakly branch groups}, studied intensively by R. Grigorchuk and his co-authors. In fact, the ISR property is not known for  several well-studied classes of groups. 

We summarize the above mentioned problems below.

 \begin{question}
Do the following classes of groups  have the ISR property assuming that the trivial amenable radical condition holds?
\begin{itemize}
%\item non-simple full groups of Bratteli diagrams. 
\item Branch or weakly branch groups.
    \item  Baumslag-Solitar groups $BS(m, n)=\langle a,b\mid a^{-1}b^ma=b^n\rangle$.
    \item Topological full groups.
    \item Free $m$-generator Burnside group $B(m, n)$ of exponent $n$.
\end{itemize}
\end{question}

A natural generalization of the ISR property would be to consider the situation when the group von Neumann algebra is replaced with the von Neumann algebra associated to a finite type factor representation.
\begin{question} Let $G$ be an infinite group and $\pi$ be its factor-representation. For $H<G$ let $L_\pi(H)$ be the von Neumann algebra generated by the operators $\{\pi(g):g\in H\}$. Under which conditions on the group $G$ and the representation $\pi$ the following is true: each invariant von Neumann subalgebra $P$ of $L_\pi(G)$ is of the form $L_\pi(H)$ for some normal subgroup $H$ of $G$. 
\end{question}

\end{document}